\documentclass[11pt]{article}
\usepackage{amsmath,amsfonts,latexsym, xspace,amsthm,graphicx, amssymb}
\usepackage{enumerate}
\usepackage{url}

\usepackage{tikz}
\tikzstyle{every node}=[circle, draw, fill=black!0, inner sep=1pt, minimum width=4pt]
\tikzstyle{label node}=[draw=none,fill=none]

\usepackage{mathtools}
\mathtoolsset{showonlyrefs}

\usepackage{chngcntr}
\counterwithin{equation}{section}

\newcommand{\al}[1]{
  \begin{align}
  #1
  \end{align}
}

\def\AND{\text{ and }}
\def\OR{\text{ or }}
\def\LRG{\mathbb{L}}
\def\SML{\mathbb{S}}

\newcommand{\ind}{\mathbf{1}}
\newcommand{\E}[1]{\mathbb{E} \left[#1\right]}

\newcommand{\Bin}[2]{\text{Bin}\left(#1, #2\right)}
\newcommand{\Var}[1]{\text{Var}\left(#1\right)}
\newcommand{\fl}[1]{\left\lfloor #1\right\rfloor}
\newcommand{\cl}[1]{\left\lceil #1\right\rceil}

\def\dist{\mathrm{dist}}
\def\alt{\mathrm{alt}}
\def\mix{\mathrm{mix}}
\def\odd{\mathrm{odd}}
\def\Po{\mathrm{Po}}
\def\rot{\mathrm{rot}}
\def\AW{\mathrm{AW}}
\def\BW{\mathrm{BW}}
\def\RM{\mathrm{RM}}
\def\EP{\mathrm{EP}}
\def\IP{\mathrm{IP}}
\def\SE{\mathrm{SE}}
\def\XY{\mathrm{XY}}

\def\AA{\mathcal{A}}
\def\BB{\mathcal{B}}
\def\CC{\mathcal{C}}
\def\DD{\mathcal{D}}
\def\EE{\mathcal{E}}
\def\FF{\mathcal{F}}
\def\GG{\mathcal{G}}
\def\LL{\mathcal{L}}
\def\MM{\mathcal{M}}

\def\QQ{\mathcal{Q}}
\def\RR{\mathcal{R}}

\def\TT{\mathcal{T}}
\def\WW{\mathcal{W}}
\def\XX{\mathcal{X}}
\def\YY{\mathcal{Y}}

\def\ol{\overline}
\def\wh{\widehat}

\def\R{\mathbb{R}}
\def\a{\alpha}
\def\b{\beta}
\def\d{\delta}
\def\D{\Delta}
\def\e{\varepsilon}
\def\f{\phi}
\def\g{\gamma}
\def\G{\Gamma}
\def\k{\kappa}
\def\th{\theta}
\def\lam{\lambda}
\def\m{\mu}
\def\r{\rho}
\def\s{\sigma}

\def\t{\tau}

\def\OM{\Omega}
\newcommand\Prob[1]{{\mbox{Pr}\left\{#1\right\}}}

\newtheorem{theorem}{Theorem}
\numberwithin{theorem}{section}
\newtheorem{proposition}[theorem]{Proposition}
\newtheorem{lemma}[theorem]{Lemma}
\newtheorem*{lemma*}{Lemma}
\newtheorem*{theorem*}{Theorem}

\newtheorem{definition}[theorem]{Definition}
\newtheorem{claim}{Claim}
\numberwithin{claim}{section}

\newcommand{\bfrac}[2]{\left({\frac{#1}{#2}}\right)}

\title{Hamilton cycles in weighted Erd\H{o}s-R\'enyi graphs}

\author{Tony Johansson\thanks{Funded by the Swedish Research Council (grant 2015-05015).} \\[.2cm]
  {\small Stockholm University} \\
  {\small Stockholm, Sweden}
}

\begin{document}
\maketitle

\begin{abstract}
Given a symmetric $n\times n$ matrix $P$ with $0 \le P(u, v)\le 1$, we define a random graph $G_{n, P}$ on $[n]$ by independently including any edge $\{u, v\}$ with probability $P(u, v)$. For $k\ge 1$ let $\AA_k$ be the property of containing $\fl{k/2}$ Hamilton cycles, and one perfect matching if $k$ is odd, all edge-disjoint. With an eigenvalue condition on $P$, and conditions on its row sums, $G_{n, P}\in \AA_k$ happens with high probability if and only if $G_{n, P}$ has minimum degree $k$ whp. We also provide a hitting time version. As a special case, the random graph process on pseudorandom $(n, d, \m)$-graphs with $\m \le d(d/n)^\a$ for some constant $\a > 0$ has property $\AA_k$ as soon as it acquires minimum degree $k$ with high probability.
\end{abstract}


\section{Introduction}

The problem of determining whether a random graph contains a Hamilton cycle, i.e. a cycle passing through every vertex exactly once, dates back to the inception of the study of random graphs. In 1960, Erd\H{o}s and R\'enyi asked whether their eponymous random graph $G_{n, p}$, obtained by including any edge independently with probability $p$, contains a Hamilton path \cite{ErdosRenyi60}. The problem was settled by Koml\'os and Szemer\'edi \cite{KomlosSzemeredi}, who showed that
$$
\lim \Prob{G_{n, p} \text{ is Hamiltonian}} = \lim \Prob{\d(G_{n, p})\ge 2},
$$
where $\d$ denotes the minimum degree of a graph. This was strengthened to a hitting time result, independently by Bollob\'as \cite{bollobas84} and by Ajtai, Koml\'os and Szemer\'edi \cite{AjtaiKomlosSzemeredi}, stated as follows. Suppose $G_{n, m}$ is an increasing sequence of random graphs, where $G_{n, m}$ is obtained by adding a uniformly chosen edge to $G_{n, m-1}$. Let $\t_2$ be the smallest $m$ for which $\d(G_{n, m}) \ge 2$. Then with high probability, $G_{n, \t_2}$ is Hamiltonian. For integers $k\ge 1$, let $\AA_k$ denote the graph property of containing $\fl{k/2}$ Hamilton cycle, as well as a matching of size $\fl{n/2}$ when $k$ is odd. Bollob\'as and Frieze~\cite{BollobasFrieze} strengthened the hitting time result above to showing that $G_{n, \t_k}\in \AA_k$ with high probability. For a more thorough history, see Frieze's recent survey on Hamilton cycles~\cite{Frieze19}.

In recent years, some attention has been turned to random subgraphs of a host graph $\G_n$. The random subgraph $\G_{n, p}$ is obtained by including any edge of $\G_n$ independently with probability $p$, and a the random graph process $\G_{n, m}$ on $\G_n$ is obtained by ordering the edges of $\G_n$ uniformly at random. An early example is the random bipartite graph $G_{n, n, p}$, obtained by letting $\G_n = K_{n, n}$. Frieze~\cite{Frieze85} determined the threshold in this case. Frieze and Krivelevich~\cite{FriezeKrivelevich02} showed that if $\G_n$ is in a certain class of pseudorandom graphs (specified below) then $\G_{n, \t_2} \in \AA_2$ whp. The author~\cite{Johansson20} showed that the same holds when $\d(\G_n) \ge (1/2+\e)n$ for some constant $\e > 0$ (and that it need not hold when $\d(\G_n) = n/2$). Alon and Krivelevich~\cite{AlonKrivelevich20} showed that $\G_{n, \t_{2k}} \in \AA_{2k}$ whp for any $k=O(1)$ in three dense classes of host graphs, which include some pseudorandom graphs and graphs with $\d(\G_n) \ge (1/2+\e)n$. 

Both \cite{FriezeKrivelevich02} and \cite{AlonKrivelevich20} consider pseudorandom graphs known as $(n, d, \m)$-graphs. A graph $\G$ is an $(n,d,\m)$-graph if it has $n$ vertices, every vertex has degree $d$, and the second largest eigenvalue of its adjacency matrix is at most $\m$ in absolute value. We strengthen both results in the following special case of our main theorem. Let $\t_{\AA_k}$ be the smallest $m$ for which $\G_{n, m}\in \AA_k$.
\begin{theorem*}[Theorem~\ref{thm:1}, pseudorandom graph case]
Let $k = O(1)$. Suppose $\G_n$ is an $(n, d, \m)$-graph with $\m \le d(d/n)^\a$ for some constant $\a > 0$. Then the random graph process on $\G_n$ has $\t_{\AA_k} = \t_k$ with high probability.
\end{theorem*}
In~\cite{FriezeKrivelevich02} it was required that $\m = o(d^{5/2} / (n\ln n)^{3/2})$, which only holds if $d \gg n^{3/4}(\ln n)^3$, while~\cite{AlonKrivelevich20} asked that $\m = O(d^2/n)$ and $d = \OM(n(\ln\ln n) / \ln n)$. This result strengthens both, with an implicit degree bound of $d = n^{\OM(1)}$ owing to the fact that $\m = \OM(d^{1/2})$ (see e.g.~\cite{Nilli91}).

Our full result concerns a more general inhomogeneous random graph. Suppose $P$ is a symmetric $n\times n$ matrix with entries $P(u, v)\in [0, 1]$. We then define a random graph $G_{n, P}$ by independently including each edge $\{u, v\}$ with probability $P(u, v)$.  If $R$ is a symmetric $n\times n$ matrix with $R(u, v) \ge 0$ for all $uv$, we also define a random graph process $G_{n, R}(t)$ as follows. Each pair $\{u, v\}$ is independently assigned a random value $E(u, v)$, exponentially distributed with rate $R(u, v)$, taken to equal $\infty$ if $R(u, v) = 0$. We let
$$
G_{n, R}(t) = ([n], \{uv : E(u, v) \le t\}).
$$
Note that $G_{n, R}(t)$ equals $G_{n, P}$ in distribution when $P(u, v) = 1-e^{-R(u, v)t}$. Note that this framework generalizes $\G_{n, p}$ and $\G_{n, m}$ (now in continuous time), re-obtained by letting $R$ be the adjacency matrix of $\G_n$. Anastos, Frieze and Gao~\cite{AnastosFriezeGao19} considered Hamiltonicity in the stochastic block model, which is $G_{n, P}$ with $P$ in a specific class of block matrices.

For vertex sets $A, B$ we let $R(A, B) = \sum_{u\in A, v\in B} R(u, v)$. Let $d_R(u) = R(u, V)$ and $d_R(A) = R(A, V)$. A key tool in our proof is the random walk induced by $R$, which jumps from $u$ to $v$ with probability $M(u, v) = R(u, v) / d_R(u)$. The transition matrix $M$ has real eigenvalues $1 = \lam_1 \ge \lam_2 \ge\dots\ge \lam_n\ge -1$ (see e.g.~\cite{LevinPeres17}), and we let $\lam(R) = \max\{|\lam_2|, |\lam_n|\}$. Let $\s(u) = d_R(u) / d_R(V)$ be the stationary distribution of the random walk, and define $\|M\| = \max_{u, v} M(u, v)$.

\begin{definition}\label{def:degrees}
Let $\RM$ be the set of rate matrices $R$ with transition matrix $M$ such that there exist constants $\a \in [0, 1/2)$, $\g > 0$, $b > 0$ such that $\lam(R) = o(1)$ and
\begin{equation}\label{eq:lamcond}
\lam(R) \le (n\|M\|)^{-\a - \g},
\end{equation}
and $R$ satisfies the following:
\begin{enumerate}[(a)]
\item\label{item:scale} there exists a $d = d(n)$ such that $d_R(u) \ge d$ for all $u$, and $d_R(V) \le bdn$.
\item\label{item:powerlaw} for any $A\subseteq V$,
\begin{equation}\label{eq:powerlaw}
\s(A) = \frac{d_R(A)}{d_R(V)} \le b\bfrac{|A|}{n}^{1-2\a},
\end{equation}
\item $\|R\| \le dn^{-\g}$.
\end{enumerate}
\end{definition}

We state our main result.
\begin{theorem}\label{thm:1}
Let $k = O(1)$. If $R\in \RM$, then whp $G_{n, R}(t)$ satisfies
$$
\t_{k} = \t_{\AA_{k}}.
$$
\end{theorem}
For any symmetric non-negative matrix $R$ on $V$, let
$$
\g_k(R) = \sum_{u\in V} d_R(u)^{k-1} e^{-d_R(u)}.
$$
Let $\RM(1)$ be the set of $R\in \RM$ with $\g_1(R) = 1$. Note that for any $x > 0$, the graphs $G_{n, R}(\t_k)$ and $G_{n, xR}(\t_k)$ are equal in distribution, so it is enough to prove Theorem~\ref{thm:1} for $R\in \RM(1)$.
\begin{theorem}\label{thm:pi}
Let $k = O(1)$. Suppose $P\in \RM$ has $\g_k(P) \to \g_k \in [0,\infty]$. Then
$$
\lim_{n\to\infty} \Prob{G_{n, P}\in \AA_{k}} = e^{-\g_k}.
$$
\end{theorem}
Note that if $P$ has constant row sums $d_P(u) = d = \ln n + (k-1)\ln\ln n + c_{n, k}$, then $\g_k = \exp\{-\lim_n c_{n, k}\}$.

\section{Proof outline}\label{sec:Hdef}

Let us first discuss the overarching proof idea. Traditionally, many proofs of Hamiltonicity in random graphs rely on finding so-called {\em booster edges} for a graph $G$, which are edge $e\notin G$ such that $G\cup \{e\}$ is closer to being Hamiltonian, typically meaning it contains a longer path than $G$ does. If random edges are added to $G$, one argues that some booster is likely to be added.

Montgomery~\cite{Montgomery19} more generally defined boosters as sets $T$ of edges whose addition gets $G$ closer to Hamiltonicity, and used sets $|T| \le 2$. A similar idea was used earlier in~\cite{FriezeJohansson17}. Booster pairs were also used by Alon and Krivelevich in their recent paper~\cite{AlonKrivelevich20}. In this paper we move to general boosters, i.e. edge sets $T$ of any (constant) size whose addition improves $G$. These are found using random alternating walks.

\subsection{Random subgraphs}\label{sec:Hintro}

For proving Hamiltonicity, the most important property of $G_{n, R}(\t_k)$ is expansion (see Lemma~\ref{lem:H} below for the definition). A major drawback of $G_{n, R}(\t_k)$ for our purposes is its high average degree, with most vertices having degree $\OM(\ln n)$. We therefore define a random subgraph $H\subseteq G_{n, R}(\t_k)$, which retains expansion and connectivity with high probability while containing only $O(n)$ edges. We do not show that $H$ itself is Hamiltonian, but the graph will be important to our proof.

Recall the construction of $G_{n, R}(t)$, in which each edge $\{u, v\}$ is included at a random time $E(u, v)$. Let $D\ge k$ be a constant integer and let
$$
T_D(u) = \inf\{t > 0 : d_t(u) \ge D\}
$$
be the random time at which $u$ attains degree $D$. Define $H(t)\subseteq G_{n, R}(t)$ by including any edge $\{u, v\}$ with $E(u, v) \le t$ and $E(u, v) \le \max\{T_D(u), T_D(v)\}$. Let $H = H(\t_k)$. In other words, an edge is included in $H$ if it is among the first $D$ edges attached to one of its endpoints in $G_{n, R}(\t_k)$.

For a graph $G = (V, E)$ and $A\subseteq V$ we let $N(A) = \{u\notin A : \{u, v\}\in E, \text{ some $v$}\}$ and $\wh N(A) = A\cup N(A)$. The following lemma is proved in Section~\ref{sec:Hprops}.
\begin{lemma}\label{lem:H}
Suppose $R\in \RM$, and let $k = O(1)$. For a large enough constant $D$, the following holds.
\begin{enumerate}[(i)]
\item (Light tail) Let $\th$ tend to infinity with $n$ and let $S_{\th}$ be the set of $u$ with $d_H(u) \ge \th$ or $d_R(u) \ge \th d$. Then with high probability, $|\wh N_H(S_{\th})| = o(n)$.
\item (Expansion) There exists a constant $\b = \b(R, k) > 0$ such that with high probability, every vertex set $A$ with $|A| < \b n$ satisfies $|N_{H}(A)| \ge k|A|$.
\end{enumerate}
\end{lemma}
Let $\SE_k(R)$ (``sparse expanders'') be the set of graphs satisfying (i) -- (ii). Note that if $G\in \SE_k$ and $\D(F) \le \ell < k$ then $G\setminus F \in \SE_{k-\ell}$ and $G\cup F\in \SE_k$. Note also that $\SE_k \subseteq \SE_\ell$.

For $t > 0$ let
\begin{equation}\label{eq:SMLdef}
\SML(t) = \{u : T_D(u) > t\} = \{u : d_t(u) < D\},
\end{equation}
and $\LRG(t) = V\setminus \SML(t)$. For $0\le t_0 \le t_1$ define a random graph $G_{n, R}^*(t_0, t_1)$ by including any edge $\{u, v\}$ with $E(u, v)\le t_1$, and any $\{u, v\}$ with one endpoint in $\SML(t_0)$ and $E(u, v) \le \t_k$. Then $G_{n, R}(t_1)\subseteq G_{n, R}^*(t_0, t_1) \subseteq G_{n, R}(\t_k)$ whenever $t_1 \le \t_k$, and $H\subseteq G_{n, R}^*(t_0, t_1)$ for any $0 < t_0\le t_1$. Suppose $t_0 < t_1 < t_2$ and define
$$
G_0 = G_{n, R}^*\left(t_0, t_0\right), \quad G_1 = G_{n, R}^*\left(t_0, t_1\right),\quad G_2 = G_{n, R}^*\left(t_0, t_2\right).
$$
For $i=0,1,2$ let $\FF_i$ denote the $\s$-algebra generated by $G_{n, R}^*(t_0, t_0)$ (including $E(u, v)$ for all edges included) and $G_{n, R}^*(t_0, t_i)$ (excluding $E(u, v)$). Then $H$ is $\FF_i$-measurable for all $i$, and the following lemma lets us jump between $G_1$ and $G_2$.

\begin{lemma}\label{lem:Gstar}
Suppose $R\in \RM(1)$. Let $0 < t_0 < t_1 < t_2$ and $G_i = G_{n, R}^*(t_0, t_i)$ for $i=0,1,2$. Let $\LL_2\in \FF_2$. Then for any set $F$ of edges
\al{
  \Prob{F\subseteq G_2 \mid \FF_1} = \prod_{\substack{\{u, v\}\in F\setminus G_1 \\ u,v\in \LRG(t_0)}} \left(1 - e^{-R(u, v)(t_2-t_1)}\right), \label{eq:add} \\
\Prob{F\subseteq G_1 \mid \{F\subseteq G_2\}\cap \LL_2} \ge \left(\frac{t_1-t_0}{t_2-t_0}\right)^{|F|}. \label{eq:remove}
}
\end{lemma}
\begin{proof}
Any edge in $G_2\setminus G_1$ is fully contained in $\LRG(t_0)$. Conditional on $G_0$ and $\SML(t_0)$, the edges $\{u, v\}\notin G_0$ with $u, v\in \LRG(t_0)$ are independent exponential random variables with individual rates $R(u, v)$, conditioned to be at least $t_0$. Then \eqref{eq:add} follows from the memoryless property of exponential random variables. Since $\|R\| = o(1)$, we have
\begin{multline}
\Prob{F\subseteq G_1 \mid \{F\subseteq G_2\}\cap \LL_2} = \prod_{\{u, v\}\in F\setminus G_0}\frac{e^{-R(u, v)t_0} - e^{-R(u, v)t_1}}{e^{-R(u, v)t_0} - e^{-R(u, v)t_2}}  \\
= \left(\frac{t_1-t_0}{t_2-t_0}\right)^{|F\setminus G_0|} \ge \left(\frac{t_1-t_0}{t_2-t_0}\right)^{|F|}.
\end{multline}
\end{proof}

\subsection{The high-level argument}\label{sec:highlevel}

Suppose $R\in \RM(1)$ and let $t_0 = 1/4, t_1 = 1/2, t_2 = 3/4$. Define $G_0, G_1, G_2$ as in the previous section. We will show that $G_2 \in \AA_k$ whp. In Section~\ref{sec:degrees} we show that $\t_k > 3/4$ whp, and since $G_2 \subseteq G_{n, R}(\t_k)$ when $\t_k > 3/4$, it follows that
\begin{equation}\label{eq:actualpunchline}
\Prob{G_{n, R}(\t_k) \in \AA_k} \ge \Prob{G_2 \in \AA_k\AND \t_k > 3/4} = 1-o(1).
\end{equation}

Say that $F$ is a $k$-graph if it can be written as the disjoint union of $F_1, \dots, F_{\fl{k/2}}$ where $F_i$ is a path or a cycle for all $i$, as well as a matching $F_0$ when $k$ is odd. Define
$$
s_k(G) = \max\{|F| : F\subseteq G \text{ a $k$-graph}\},
$$
so that $G\in \AA_k$ if and only if $s_k(G) = \fl{kn / 2}$.

For $i=1,2$ let $\MM_i^\ell$ be the event that $s_k(G_i) = \ell$. Recall the definition of $\SML(t)$ from \eqref{eq:SMLdef}, and let
$$
\LL = \{H\in \SE_k\} \cap \{|\SML(t_0)| = o(n)\},
$$
and note that $\LL$ is $\FF_0$-measurable. Then $\MM_i^\ell\cap \LL \in \FF_i$ for $i=1,2$, and Lemma~\ref{lem:Gstar} gives
$$
\Prob{\MM_1^\ell \mid \MM_2^\ell \cap \LL} \ge \min_{\substack{F \text{ $k$-graph}}} \Prob{F\subseteq G_1\mid \LL\cap \{F\subseteq G_2\}} \ge \bfrac12^{kn/2}.
$$
Then for any $\ell < \fl{kn / 2}$,
\al{
  \Prob{\MM_2^\ell \cap \LL} & = \frac{\Prob{\MM_1^\ell \cap \MM_2^\ell\cap \LL}}{\Prob{\MM_1^\ell \mid\MM_2^\ell\cap\LL}} \le \frac{\Prob{\MM_2^\ell \mid \MM_1^\ell\cap \LL}}{2^{-kn/2}}. \label{eq:bayes}
}
Suppose we are able to prove that for any $\ell < \fl{kn/2}$,
\begin{equation}\label{eq:numeratorbound}
\Prob{\MM_2^\ell \ \middle|\ \MM_1^\ell\cap \LL} \le e^{-\OM(n\sqrt{\ln n})}. 
\end{equation}
Then plugging \eqref{eq:numeratorbound} into \eqref{eq:bayes} gives
\al{
  \Prob{G_2 \notin \AA_k} & \le \Prob{\ol\LL} + \sum_{\ell < \fl{kn/2}} \Prob{\MM_2^\ell\cap\LL} \\
  & \le \Prob{\ol\LL} + kne^{-\OM(n\sqrt{\ln n}) + O(n)} = \Prob{\ol \LL} + o(1).
}
In Lemma~\ref{lem:degreebound} we will show that $|\SML(t_0)| = o(n)$ whp. Together with Lemma~\ref{lem:H}, this shows that $\LL$ occurs whp. As noted in \eqref{eq:actualpunchline}, this together with a proof that $\t_k > 3/4$ whp (Lemma~\ref{lem:threshold} below) shows that $G_{n, R}(\t_k) \in \AA_k$ whp.

It remains to prove \eqref{eq:numeratorbound}. The graph $G_1$ contains $H$ by construction. If $\MM_1^\ell\cap \LL$ holds then $G_1$ contains some $k$-graph $F$ of size $\ell$. If $k$ is even, suppose $F = F_1\cup \dots F_{k/2}$ for some disjoint paths and Hamilton cycles $F_i$. Suppose without loss of generality that $|F_1| < n$. Then $G_1$ contains the graph $G = H\cup F_1 \setminus (F_2\cup \dots \cup F_{k/2}) \in \SE_2$. Let $G^p$ be the graph obtained by independently adding any edge $\{u, v\}$ to $G$ with probability $p(u, v)$, where $p$ is some symmetric function. Letting$$
p(u, v) = \left\{\begin{array}{ll}
0, & u\in \SML(t_0)\OR v\in \SML(t_0), \\
0, & \{u, v\} \in F_2\cup\dots\cup F_{k/2}, \\
\frac13R(u, v), & \text{ otherwise},
\end{array}\right.
$$
we have $G^p\subseteq G_2$ by Lemma~\ref{lem:Gstar}. If $s_2(G^p) > s_2(G)$ then $s_k(G_2) \ge s_k(G_1\cup G^p)  > s_k(G_1)$, since $G^p$ is disjoint from $F_2\cup\dots\cup F_{k/2} \subseteq G_1$.

If $k$ is odd and $F = F_1\cup\dots\cup F_{(k+1)/2}$, assume without loss of generality that $\D(F_1) \le i$ and $|F_1| < in/2$ for some $i\in \{1, 2\}$. Then $G = H\cup F_1 \setminus (F_2\cup\dots\cup F_{(k+1)/2}) \in \SE_1$, and $s_i(G^p) > s_i(G)$ implies $s_k(G_2) > s_k(G_1)$.

So, \eqref{eq:numeratorbound} follows from the following lemma.
\begin{lemma}\label{lem:thebiglemma}
Let $i\in \{1, 2\}$. Suppose $R\in \RM(1)$ and $G\in \SE_i(R)$ with $s_i(G) < in/2$, and suppose $p(u, v) \ge R(u, v) / 3$ for all $\{u, v\}\notin E$ where $\sum_{\{u, v\}\in E} R(u, v) = o(n\ln n)$. Then
\al{
  \Prob{s_i(G^p) = s_i(G)} \le e^{-\sqrt{n\ln n}}.
}
\end{lemma}

The remainder of the paper is mainly devoted to proving Lemma~\ref{lem:H} (Section~\ref{sec:Hprops}) and Lemma~\ref{lem:thebiglemma} (Sections~\ref{sec:alternatingwalks} through~\ref{sec:alternatingproofs}).

\section{Preliminaries}\label{sec:prelims}

We state some preliminary, for the most part standard, results, and leave the proofs for Section~\ref{sec:prelimproofs}.

\subsection{Degrees}\label{sec:degrees}

Recall that $d_R(u) = \sum_v R(u,v)$. We will often assume that $\g_1(R) = 1$, and note that this implies that $d = \Theta(\ln n)$ where $d = \min d_R(u)$.

\begin{lemma}\label{lem:degreebound}
Let $R\in \RM(1)$.
\begin{enumerate}[(i)]
\item For any integer $D \ge 1$, $u\in V$ and $S\subseteq V$ with $|V\setminus S| = O(1)$, and $t = \OM(1)$,
$$
\Prob{e_t(u, S) < D} = \exp\{-td_R(u) + O(\ln d_R(u))\}.
$$
\item Let $\SML(t)$ denote the set of vertices in $G_{n, R}(t)$ with degree less than $D$. If $t = \OM(1)$ then $|\SML(t)| = o(n)$ whp.
\end{enumerate}
\end{lemma}

Recall the definition $\g_k(P) = \sum_u e^{-d_P(u)} d_P(u)^{k-1}$.
\begin{lemma}\label{lem:threshold}
Let $k\ge 1$. Suppose $P\in \RM$ has $\g_k(P) \to \g_k \in [0,\infty]$. Then
$$
\lim_{n\to \infty} \Prob{\d(G_{n, P}) \ge k} = e^{-\g_k}.
$$
If $R\in \RM(1)$ and $\e \gg \frac{\ln \ln n}{\ln n}$, then $G_{n, R}(t)$ has $1-\e < \t_k < 1 + \e$ whp.
\end{lemma}

We also note the following simple consequence of conditions~(\ref{item:scale}) and~(\ref{item:powerlaw}) of Definition~\ref{def:degrees}.
\begin{lemma}\label{lem:sublinear}
Suppose $R\in \RM$ has stationary distribution $\s$, and $A\subseteq V$. Then $\s(A) = o(1)$ if and only if $|A| = o(n)$.
\end{lemma}

\subsection{The expander mixing lemma}

For a matrix $A$ indexed by $V$ and sets $S, T\subseteq V$ we write
$$
A(S, T) = \sum_{u\in S, v\in T} A(u, v),
$$
noting that pairs $(u,v)$ with $u,v\in S\cap T$ are counted twice. We will use the following version of the well-known Expander Mixing Lemma~\cite{AlonChung88}.
\begin{lemma}\label{lem:EML}
Suppose $R\in RM$ has transition matrix $M$ with stationary distribution $\s$. Then for any $A, B\subseteq V$,
\al{
  M(A, B) = |A|\s(B) + O\left(\lam(R)\sqrt{n|A|\s(B)}\right).\label{eq:fullEML}
}
In particular, the following holds.
\begin{enumerate}[(i)]
\item $R(A, B) = \OM(dn)$ for any $A, B\subseteq V$ with $|A|, |B| = \OM(n)$.
\item\label{item:kappa} There exists a constant $c > 0$ such that 
$$
\left|\left\{u\in A : M(u, A) \ge \bfrac{|A|}{n}^c\right\}\right| = o(|A|) \quad \text{for all $|A| \le n/2$}.
$$
\end{enumerate}
\end{lemma}

\subsection{Probabilistic bounds}

We consider the following lemma well-known, and state it without proof.
\begin{lemma}\label{lem:exprank}
Suppose $X_1,\dots,X_m$ are independent exponential random variables with finite respective rates $r_1,\dots,r_m > 0$. Let $r = r_1+\dots + r_m$ and suppose $r_i \le \e r$ for all $i$, for some $\e > 0$. Let $X_{(D)}$ be the $D$--th smallest value in the family. Then for $D \le 1/2\e$,
$$
\Prob{X_i \le X_{(D)}} \le 2D\frac{r_i}{r}.
$$
\end{lemma}
We will also use the following Chernoff bounds.
\begin{lemma}\label{lem:chernoff}
Suppose $X$ is a finite set and let $\s_x \in [0, 1]$ for all $x\in X$. Let $\m = \sum_{x\in X} \s_x$.
\begin{enumerate}[(i)]
\item Suppose $S\subseteq X$ is a random set obtained by including any $x\in X$ independently with probability $\s_x$. If $\f \le \frac12 \m$, then
\al{
  \Prob{|S| \le \f} & \le \exp\left\{-\frac{\m}{8}\right\}.\label{eq:chernoff2}
}
\item Suppose $T\subseteq X$ is a random set with $\Prob{A\subseteq T} \le \prod_{x\in A} \s_x$ for all $A\subseteq X$. If $\f \ge 5\m$, then
\al{
  \Prob{|T| \ge \f} \le \bfrac{\m}{\f}^{\f/2}. \label{eq:chernoff4}
}
\end{enumerate}
\end{lemma}
\begin{proof}
Note that the condition on $T$ implies that $\E{|T|^\ell} \le \E{|S|^\ell}$ for all $\ell\ge 0$. The bounds then follow by standard methods, see e.g. \cite[Section 21.4]{FriezeKaronski}.
\end{proof}

\subsection{A matrix lemma}

Suppose $I$ is a totally ordered set, $A$ a set and ${\bf a} = (a_1,\dots,a_\ell)\in A^\ell$ a sequence in $A$. Say that a function $f : I\to A$ {\em respects} ${\bf a}$ if the sequence $(f(i))_{i\in I}$ is a subsequence of ${\bf a}$.
\begin{lemma}\label{lem:observation}
Suppose $\t : I\times J \to A$ and ${\bf a} = (a_1,\dots,a_\ell)\in A^\ell$, $\ell \ge 1$. Suppose $I$ can be totally ordered so that $i\mapsto \t(i, j)$ respects ${\bf a}$ for each $j\in J$. Let $\pi_I, \pi_J$ be finite measures on $I, J$, respectively. Then there exist $S\subseteq I, T\subseteq J$ with $\pi_I(S) \ge \pi_I(I) / \ell$ and $\pi_J(T) \ge \pi_J(J) / \ell$, such that $\t$ is constant on $S\times T$.
\end{lemma}

\subsection{Mixing in simple random walks}\label{sec:onestep}

Let $M$ be a transition matrix with stationary distribution $\s$ with $\s(u) > 0$ for all $u\in V$. For a probability measure $\pi$ on $V$, define
$$
\m_\s(\pi) = \sqrt{\left(\sum_{u\in V} \frac{\pi(v)^2}{\s(v)}\right) - 1}.
$$
Note that $\m_\s(\pi) \ge 0$ with equality if and only if $\pi = \s$. We use $\m_\s$ as a measure of distance from stationarity.

\begin{lemma}\label{lem:norm}
Suppose $R\in \RM$ has transition matrix $M$. Let $\pi_0$ be a probability measure on $V$, and define $\pi_1(v) = \sum_u \pi_0(u)M(u, v)$. Then
\al{
  \m_\s(\pi_1) \le \lam(R) \m_\s(\pi_0).
}
\end{lemma}

\section{On alternating walks}\label{sec:alternatingwalks}

Suppose $F = (V, E)$ is a graph, and $R$ is a rate matrix on $V$. A walk $W = (w_0,w_1,\dots,w_\ell)$ on $V$ is {\em $F$-alternating} if $\{w_i, w_{i+1}\}\in F$ for all odd $i$, and {\em strictly $F$-alternating} if also $\{w_i, w_{i+1}\}\notin F$ for all even $i$.

We need a way to measure the size of a family of $F$-alternating walks. This will be slightly cumbersome to define. Firstly, for any walk $W = (w_0,\dots,w_\ell)$ define
$$
R_\alt[W] = \prod_{j = 0}^{\fl{\ell/2}-1} R(w_{2j}, w_{2j+1}).
$$
For a family of walks $\WW$ let $R_\alt[\WW] = \sum_{W\in \WW}R_\alt[W]$.

For any edge set $E$ with $E\cap G = \emptyset$ we also define
$$
R_G[E] = \prod_{\{u, v\}\in E} R(u, v).
$$
If $E\cap G \ne \emptyset$ let $R_G[E] = 0$. If $\EE$ is a family of edge sets, let $R_G[\EE] = \sum_{E\in \EE} R_G[E]$.

For a walk $W$ let $\odd(W)$ be the set of edges $\{\{w_{2i}, w_{2i+1}\} : i \ge 0\}$. Note that if $W$ is a walk which repeats no vertex and is strictly $G$-alternating, then $R_G[\odd(W)] = R_\alt[W]$. If $E$ is an edge set of size $r$, then the number of walks $W$ with $\odd(W) = E$ is $2^rr!$. We conclude that if $\WW$ is a family of non-repeating strictly $G$-alternating walks of length $2r-1$ and $\odd(\WW) = \{\odd(W) : W\in \WW\}$, then
\begin{equation}\label{eq:Galt}
R_G[\odd(\WW)] \ge \frac{1}{2^rr!} R_\alt[\WW].
\end{equation}

For a walk $W = (w_0,\dots,w_i)$ let $f(W) = w_i$ denote its final vertex. For a family $\WW$ of walks and $v\in V$ let $\WW^{\to v}$ be the set of $W\in \WW$ with $f(W) = v$. Let $(W, v_1,\dots,v_j) = (w_0,\dots,w_i, v_1,\dots,v_j)$, and for two walks $W_1 = (w_0,\dots,w_i)$ and $W_2 = (w_0',\dots,w_j')$ write
$$
W_1\circ W_2 = (w_0,\dots,w_i, w_j',\dots,w_0').
$$

\subsection{Mixing for alternating walks}\label{sec:altmix}

We define a random walk on $V$ as a probability measure $\pi$ on the set $V^\infty$ of infinite walks on $V$. For walks $W$ of length $\ell$, write $\pi(W) = \pi(\WW(W))$ where $\WW(W)$ is the family of walks agreeing with $W$ for the first $\ell$ steps. Define
$$
\pi(w_{j+1}\mid w_0,\dots,w_j) = \frac{\pi(w_0,\dots,w_{j+1})}{\pi(w_0,\dots,w_j)}
$$
whenever $\pi(w_0,\dots,w_j) > 0$. Define $\pi(v\mid W) = 0$ when $\pi(W) = 0$.

Suppose $G$ is a graph and $R$ a rate matrix. Recall the definition of $\wh N(A)$ from Section~\ref{sec:Hintro}. Starting at some (possibly random) initial point $w_0$, say that a random walk $\pi$ is an {\em $(R, G)$-alternating random walk} if for all $i\ge 0$,
\al{
  \pi(w_{2i + 1} \mid w_0,\dots,w_{2i}) & = M(w_{2i}, w_{2i+1}), \\
  \pi(w_{2i + 2} \mid w_0,\dots,w_{2i+1}) & = 0,\quad w_{2i+2}\notin \wh N_G(w_{2i+1}).
}
We let $\pi_i$ denote the measure on $V$ induced by $w_i$. A special case is the {\em simple, lazy $(R, G)$-alternating random walk} $\pi_{G}$ defined by
$$
\pi_G(w_{2i} \mid w_0,\dots,w_{2i-1}) = \frac{1}{d_G(w_{2i-1}) + 1}, \quad w_{2i} \in \wh N_G(w_{2i-1}),
$$
for all $i\ge 1$. If the initial vertex $x$ is specified, we denote the measure by $\pi_{G, x}$.

Note that if $W = (w_0,\dots,w_j)$ is a $G$-alternating walk and $\pi$ a random $(R, G)$-alternating walk, then
\begin{equation}
\pi(W\mid w_0) \le \prod_{i=0}^{\cl{j/2}-1}\frac{R(w_{2i}, w_{2i+1})}{d_R(w_{2i})}\le \frac{R_\alt[W]}{d^{\cl{j/2}}}. \label{eq:pitoR}
\end{equation}

For rate matrices $R$ define
\begin{equation}\label{eq:mixdef}
\mix(R) = \cl{2 - 2\frac{\ln (n\|M\|)}{\ln \lam(R)}},
\end{equation}
and note that $\mix(R) = O(1)$ for $R\in \RM$ since then $\lam(R) = (n\|M\|)^{-\OM(1)}$. The name is in reference to the following lemma.
\begin{lemma}\label{lem:weightmix}
Suppose $R \in \RM$, $\D(F)\le 2$, and that $\pi$ is an $(R, F)$-alternating random walk. Suppose $\th \le \lam(R)^{-1/4}$ tends to infinity with $n$, and that $j\ge \mix(R)$. Suppose $c > 0$ is a constant. There exists a constant $\r > 0$ such that if $\WW$ is a family of $F$-alternating walks $W$ of length $2j$, such that $d_R(v) < \th d$ for all $v\in W$, and $\pi(\WW) \ge  c$, then there exists a vertex set $U_{2j}$ of size at least $\r n$  such that
\al{
  \pi(\WW^{\to u}) \ge \frac{\r}{n}, \quad \text{for all $u\in U_{2j}$}.
}
\end{lemma}

Say that a measure $\m$ on $V$ is {\em near-uniform} if there exists a constant $c > 0$ such that $\m(v) \le c/n$ for all $v$. Say that a random walk $\pi$ is near-uniform if $\pi_0$ is near-uniform. Note that for a near-uniform $\pi$, \eqref{eq:pitoR} shows that for any family of walks $\WW$ of length $j$,
\begin{equation}\label{eq:mixpitoR}
\pi(\WW) = \sum_{u} \pi_0(u) \pi(\WW \mid u) \le \frac{c}{n} \frac{R_\alt[\WW]}{d^{\cl{j/2}}}
\end{equation}
For $S\subseteq V$ let $\t(S)$ be the random time at which a walk first visits $S$, and let $\t_\odd(S)$ be the first odd index for which it occurs.
\begin{lemma}\label{lem:hitlargeset}
Suppose $R\in \RM$ and that $G$ is light-tailed. Suppose $\pi$ is a near-uniform random $(R, G)$-alternating walk.
\begin{enumerate}[(i)]
\item\label{item:nohit} If $j\ge 0$ is constant and $|S| = o(n)$, then $\pi(\t(S) \le 2j) = o(1)$.
\item\label{item:dohit} If $|S| = \OM(n)$ then $\pi(\t(S) \le 1) = \OM(1)$.
\end{enumerate}
\end{lemma}
Lemma~\ref{lem:hitlargeset}~(\ref{item:nohit}) is particularly interesting for $S = S_\th$, which has $|S_\th| = o(n)$ if $R\in \RM$ and $G$ is light-tailed. We let $\DD_{j}^\th$ denote the event that $\t(S_\th) \le j$, and let $Z_j^\th(c)$ be the corresponding vertex set. Then $|Z_j^\th(c)| = o(n)$.

\begin{lemma}\label{lem:improper}
Suppose $R\in \RM$ and that $G$ is light-tailed. Let $j\ge 0$. Let $\CC_j^\th$ be the set of $G$-alternating walks $W = (w_0,\dots,w_j)$ with $d_G(w_i) < \th$ and $d_R(w_i) < \th d$ for all $i$, such that either (a) $|\{w_0,\dots,w_j\}| < j+1$ or (b) $W$ is not strictly $G$-alternating. If $\th^{j}\|M\| = o(1)$ then
$$
R_\alt[\CC_j^\th] = o(nd^{\cl{j/2}}).
$$
\end{lemma}

\subsection{Matchings and augmenting walks}\label{sec:matchings}

Suppose $G$ is a graph on an even number $n$ of vertices. Let $s_1(G) \le \fl{n/2}$ denote the size of the largest matching in $G$. Suppose $s_1(G) < \fl{n/2}$, and let $\FF_1(G)$ be the family of matchings $F$ with $|F| = \m(G)$. For $F\in \FF_1(G)$ let $\IP_F$ be the set of vertices isolated by $F$, i.e. vertices $x$ such that $d_F(x) = 0$. Let $\IP_2 = \IP_2(G)$ be the set of vertex pairs $\{x, y\}$  such that $\{x, y\}\subseteq \IP_F$ for some $F\in \FF_1(G)$.

\begin{lemma}\label{lem:IPk}
Suppose $G\in \SE_1$ with $s_1(G) < \fl{n / 2}$. Then $|\IP_2| \ge (\b n)^2/ 2.$
\end{lemma}
\begin{proof}
Let $F\in \FF_1(G)$ with and $x\in \IP_F$. Let $Y$ be the set of vertices $y$ such that there exists an $F$-alternating walk $(x = w_0,\dots,w_{2j} = y)$ of even length, with $x\in Y$. Then $N_G(Y) = N_F(Y)$. Indeed, suppose $v\in N_G(Y)\setminus N_F(Y)$, and let $W = (x,\dots,y)$ be an even-length $F$-alternating walk such that $yv\in G$. If $v\in \IP_F$ then $W+v$ is an augmenting walk, contradicting the maximality of $F$. If $v\notin \IP_F$ then $vw\in F$ for some $w$, and the walk $(W, v, w)$ shows that $w\in Y$, contradicting $v\notin N_F(Y)$. It follows that $|N_G(Y)| \le |N_F(Y)| \le |Y| - 1$. Since $G$ expands, we then have $|Y| \ge \b n$.

The same argument shows that for every $y\in Y$, there is a set $|X_y| \ge \b n$ such that $\{x, y\}\in \IP_2$ for each $x\in X_y$. We conclude that $|\IP_2| \ge (\b n)^2/2$.
\end{proof}

For a graph $G$, integer $r \ge 1$, and $\th$ tending to infinity, let $\TT_r^\th(G)$ be the family of edge sets $|T| = r$ such that no $e\in T$ is contained in $G$ or incident to $S_\th$, and $s_1(G \cup T) > s_1(G)$.
\begin{proposition}\label{prop:augmentingwalks}
Suppose $R\in \RM$ and $G\in \SE_1$ with $s_1(G) < \fl{n/2}$. Then there exists an $r\le 2\mix(R) + 1$ such that if $\th$ tends to infinity sufficiently slowly, then
$$
R_G[\TT_r^\th] = \OM(nd^{r}).
$$
\end{proposition}

\begin{proof}
Let $\AW_{2r-1}^\th(G)$ be the set of walks $(w_0,\dots,w_{2r-1})$ which (a) repeat no vertex, (b) avoid $S_\th$, (c) are $F$-alternating with $w_0, w_{2r-1}\in \IP_F$ for some $F\in \FF_1(G)$, and (d) are strictly $G$-alternating. Then $\odd(\AW_{2r-1}^\th(G)) \subseteq \TT_r^\th$, and \eqref{eq:Galt} shows that it is enough to prove that $R_\alt[\AW_{2r-1}^\th(G)] = \OM(nd^r)$ for some $r\le 2\mix(R) + 1$ and $\th$.

For $F\in \FF_1(G)$ let $\AA_i(F)$ be the set of $F$-alternating walks $(w_0, w_1,\dots,w_i)$ with $w_0\in \IP_F$ and $w_i \notin \IP_F$. Let $(w_0,\dots,w_i)\in \BB_i(F)$ if $(w_0,\dots,w_{i-1})\in \AA_{i-1}(F)$ and $w_i \in \IP_F$. Let $\AA_i, \BB_i$ be the walks which are in $\AA_i(F), \BB_i(F)$ for some $F\in \FF_1(G)$, respectively. Note that $\BB_i\setminus (\CC_i\cup \DD_i^\th) \subseteq \AW_{i}^\th(G)$.

Suppose $x\in \IP_F$ for some $F\in \FF_1(G)$. Let $\pi_{F, x}$ be the simple, lazy $(R, F)$-alternating random walk with starting vertex $x$, as defined in Section~\ref{sec:altmix}. Then for all $i\ge 0$, it holds that $\pi_{F, x}(\AA_{2i+1} \cup \BB_{2i + 1}) = \pi_{F, x}(\AA_{2i})$ and $\pi_{F, x}(\AA_{2i+2}) \ge \frac12\pi_{F, x}(\AA_{2i+1})$. For any $j\ge 1$, we conclude that there is a constant $c_j$ such that either $\pi_{F, x}(\BB_{2i-1}) \ge c_j$ for some $1\le i < j$, or $\pi_{F, x}(\AA_{2j}) \ge c_j$. We set $j = \mix(R)$, as defined in \eqref{eq:mixdef}.

For each pair $\{x, y\}\in \IP_2$ pick some $F(x, y)\in \FF_1(G)$ with $\{x, y\}\in \IP_F$. Define a random $(R, G)$-alternating walk $\pi$ by
$$
\pi = \sum_{\{x, y\}\in \IP_2} \frac{1}{|\IP_2|}\left(\frac{\pi_{F(x, y), x}}{2} + \frac{\pi_{F(x, y), y}}{2}\right).
$$
In other words, we pick a pair $\{x, y\} \in \IP_2$ uniformly at random, then pick one of $x, y$ as our starting point with probability $1/2$, and run the simple, lazy $(R, F(x, y))$-alternating random walk. Note that $\pi$ is near-uniform, as $|\IP_2| = \OM(n^2)$.

{\bf The one-sided case.} Suppose $\pi(\BB_{2i-1}) = \OM(1)$ for some $i < j$. Since $\BB_{2i-1} \setminus (\CC_{2i-1}\cup \DD_{2i-1}^\th) \subseteq \AW_{2i-1}^\th$, Lemmas~\ref{lem:hitlargeset}~(\ref{item:nohit}) and \ref{lem:improper} imply
\al{
  \pi(\AW_{2i-1}^\th(G)) & \ge \pi(\BB_{2i-1}) - \pi(\CC_{2i-1}) - \pi(\DD_{2i-1}^\th) = \OM(1).
}
Then \eqref{eq:mixpitoR} gives $R_\alt[\AW_{2i-1}^\th(G)] = \OM(nd^i)$. 

{\bf The two-sided case.} Suppose $\pi(\BB_{2i-1}) < c_j$ for all $i < j$. Let $\AA_{2j}^\th$ be the set of walks in $\AA_{2j}$ which avoid $\DD_{2j}^\th$, and let
$$
\XY = \left\{\{x, y\}\in \IP_2 : \pi_{F(x, y), x}(\AA_{2j}^\th) \ge \frac{c_j}{2}\AND \pi_{F(x, y), y}(\AA_{2j}^\th) \ge \frac{c_j}{2}\right\}.
$$
Then $|\XY| = \OM(n^2)$. Indeed, $\pi(\DD_{2j}^\th) = o(1)$ by Lemma~\ref{lem:hitlargeset}~(\ref{item:nohit}), so
\al{
  c_j-o(1) \le \pi(\AA_{2j}^\th) \le \frac{|\ol \XY|}{|\IP_2|}\frac{c_j}{2} + \frac{|\XY|}{|\IP_2|} \le \frac{c_j}{2} + \frac{|\XY|}{|\IP_2|}.
}

Fix some $\{x, y\}\in \XY$ and let $F = F(x, y)$. Let $\AA_{x, y} \subseteq \AA_{2j}^\th$ be the set of $F$-alternating walks in $\AA_{2j}^\th$ originating at $x$. Note that $\AA_{x, y}\circ \AA_{y, x}\subseteq \BB_{4j+1}\setminus \DD_{4j+1}^\th$. We have
\al{
  R_\alt[\AA_{x, y} \circ \AA_{y, x}] & = \sum_{u, v} R_\alt[\AA_{x, y}^{\to u}] R(u, v) R_\alt[\AA_{y, x}^{\to v}].\label{eq:glue}
}
By Lemma~\ref{lem:weightmix} there exists a constant $\r > 0$ and a set $|U_x| \ge \r n$ such that $\pi_{F, x}(\AA_{x, y}^{\to u}) \ge \r/n$ for all $u\in U_x$,
and by \eqref{eq:pitoR} we have
$$
R_\alt[\AA_{x, y}^{\to u}] \ge d^j \pi_{F, x}(\AA_{x, y}^{\to u}) = \OM\bfrac{d^j}{n}.
$$
Likewise, $R_\alt[\AA_{y, x}^{\to v}] = \OM(d^j/n)$ for all $v\in U_y$ where $|U_y| \ge \r n$. Then \eqref{eq:glue} and Lemma~\ref{lem:EML} imply
\al{
  R_\alt[\AA_{x, y}\circ \AA_{y, x}] \ge R(U_x, U_y) \OM\bfrac{d^j}{n}^2 = \OM\bfrac{d^{2j+1}}{n}. \label{eq:glue2}
}
Since $\AA_{x, y}\circ \AA_{y, x}\setminus \CC_{4j+1} \subseteq \AW_{4j+1}^\th(G)$, Lemma~\ref{lem:improper} gives
\al{
  R_G[\AW_{4j+1}^\th(G)] & \ge \left(\sum_{(x, y)\in \XY} R_\alt[\AA_{x, y}\circ \AA_{y, x}] \right) - R_\alt[\CC_{4j+1}] = \OM(nd^{2j+1}).
}
\end{proof}

\subsection{Paths and boosters}\label{sec:posa}

We view paths $P = (x = v_0,\dots,v_\ell = y)$ as being directed from $x$ to $y$, and for any $P$ define a total ordering $\le_P$ of $V$ by $v_i\le_P v_j$ whenever $i\le j$, and $u\le_P v$ whenever $u\in P, v\notin P$ (arbitrarily ordering the vertices not on $P$).

Suppose ${\bf z} = (z_1, z_2,\dots,z_\ell)$ is a sequence of distinct vertices on $P$. Define $\t({\bf z})$ as the permutation of $[\ell]$ for which $z_{\t(1)} \le_P z_{\t(2)} \le_P \dots \le_P z_{\t(\ell)}$. If ${\bf z}$ repeats a vertex or contains a vertex not on $P$, define $\t({\bf  z}) = \perp$. For a pair of $P$-alternating walks $(X, Y)$ with $X = (x,x_1,\dots,x_i)$ and $Y = (y,y_1,\dots,y_{2j})$, let $\t(X, Y) = \perp$ if $X$ has odd length and $f(X) \in N_P(Y)$, and $\t(X, Y) = \t(y_1,\dots,y_{2j}, x_1,\dots,x_i)$ otherwise (note that $\t = \perp$ is possible in this case as well). See Figure~\ref{fig:tau}.









\begin{figure}[b]
\begin{center}
\begin{tikzpicture}
\node (x) at (0, 0) {};
\node (y) at (10, 0) {};

\node[label node] at (0, -.3) {$x$};
\node[label node] at (10, -.3) {$y$};

\node[fill=black] (y1) at (3, 0) {};
\node[label node] at (3, -.3) {$y_1$};
\node[fill=black] (y2) at (3.5, 0) {};
\node[label node] at (3.5, -.3) {$y_2$};
\node[fill=black] (y3) at (8, 0) {};
\node[label node] at (8, -.3) {$y_3$};
\node[fill=black] (y4) at (7.5, 0) {};
\node[label node] at (7.5, -.3) {$y_4$};

\node[fill=black] (x1) at (5, 0) {};
\node[label node] at (5, -.3) {$x_1$};
\node[fill=black] (x2) at (4.5, 0){};
\node[label node] at (4.5, -.3) {$x_2$};

\node[label node] at (3, -.7) {$[1]$};
\node[label node] at (3.5, -.7) {$[2]$};
\node[label node] at (8, -.7) {$[3]$};
\node[label node] at (7.5, -.7) {$[4]$};
\node[label node] at (5, -.7) {$[5]$};
\node[label node] at (4.5, -.7) {$[6]$};

\draw (x) -- (y);

\draw[bend right,line width=1pt] (y) to (y1);
\draw[line width=1pt] (y1) to (y2);
\draw[bend left,line width=1pt] (y2) to (y3);
\draw[line width=1pt] (y3) to (y4);

\draw[bend left, line width=1pt] (x) to (x1);
\draw[line width=1pt] (x1) to (x2);
\end{tikzpicture}
\end{center}
\caption{Two walks $X, Y$ with $\t(X, Y) = (126543)$. If $X' = (X, x_3)$, then $\t(X', Y)$ will take the following values in order as $x_3$ increases from $x$ to $y$ and out of $P$: $\perp$, $(7126543)$, $\perp$, $(1276543)$, $\perp$, $(1265743)$, $\perp$, $(1265437)$, $\perp$. This sequence is the same for any compatible $X, Y$ with the same $\t(X, Y)$, though some values may be skipped.}
\label{fig:tau}
\end{figure}
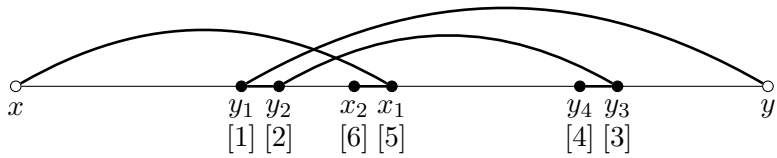

\subsubsection{Path rotations}\label{sec:rotations}

Suppose $P = (v_0,\dots,v_\ell)$ is a path of length $\ell$, and let $e = \{v_\ell, v_i\}$ be an edge with $0 < i < \ell-1$. Then $P\triangle (v_\ell, v_i, v_{i+1}) = (v_0,\dots,v_i,v_\ell,v_{\ell-1},\dots,v_{i+1})$ is also a path of length $\ell$. We say that $(v_\ell, v_i, v_{i+1})$ is a rotation walk of length 2.

In general, suppose $P$ is a path with endpoints $x$ and $y$, and suppose $W$ is a non-repeating even-length strictly $P$-alternating walk starting at $y$ and avoiding $x$. If $P\triangle W$ is a path, we say that $W$ is a {\em rotation walk} for $P$. Let $\RR_{2i}(P, y)$ be the set of rotations walks for $P$ of length $2i$ with starting point $y$. Let $\RR_{2i+1}(P, y)$ be the set of walks $(w_0,\dots,w_{2i+1})$ with $(w_0,\dots,w_{2i})\in \RR_{2i}(P, y)$ and $w_{2i+1}\in P$ with $\dist_P(w_{2i+1}, \{w_0,\dots,w_{2i}, x\}) > 1$.

Suppose $W = (w_0,\dots,w_{2i+1})\in \RR_{2i+1}(P, y)$ for some $i \ge 0$. Then there is a unique vertex $v \in N_P(f(W))$ for which $(W, v) \in \RR_{2i+2}(P, y)$, and we define $r_P(W)$ as the path $(W, v)$. Note that $\t(r_P(W))$ is fully determined by $\t(W)$, as $v$ is the immediate successor of $u$ along the path $P\triangle W$, viewed as going from $x$ to $f(W)$.

We will need to consider pairs of rotation walks starting at $x$ and $y$. Let $\AA_i(P, x)$ be the set of $P$-alternating walks of length $i$ starting at $x$. For $X\in \AA_i(P, x)$ and $Y\in \RR_{2j}(P, y)$, say that $(X, Y)$ is a {\em compatible pair} if $X\in \RR_i(P\triangle Y, x)$, no vertex appears twice in $X\cup Y$, and $\dist_P(f(X), Y) > 1$ if $X$ has odd length. For a compatible pair $(X, Y)$ let $r_{P, Y}(X) = r_{P\triangle Y}(X)$ be the unique walk $(X, v)$ for which $(r_P(X), Y)$ is compatible. Note that $\t(r_{P, Y}(X), Y)$ is fully determined by $\t(X, Y)$.

We summarize this with a lemma. For families of walks $\XX$ and $\YY$, say that $(\XX, \YY)$ is compatible if every $(X, Y)\in \XX\times \YY$ is compatible, and there exists a permutation $\t_0$ such that $\t(X, Y) = \t_0$ for all $(X, Y)\in \XX\times \YY$. 
\begin{lemma}\label{lem:orient}
Let $i, j \ge 0$. Suppose $\XX\subseteq \AA_{2i+1}(P, x)$ and $\YY \subseteq \RR_{2j}(P, y)$ are such that $(\XX, \YY)$ is compatible. Let $r_{P, \YY}(\XX) = \{r_{P, Y}(X) : X\in \XX, Y\in \YY\} \subseteq \AA_{2i+2}(P, x)$. Then $(r_{P, \YY}(\XX), \YY)$ is compatible.
\end{lemma}
\begin{proof}
This follows from the fact that $\t(r_{P, Y}(X), Y)$ is fully determined by $\t(X, Y)$, which is constant on $\XX\times \YY$.
\end{proof}
The relevance of compatible walks is this: if $(\XX, \YY)$ is a compatible pair of families of even-length walks, then $P\triangle (X\circ Y)$ is a cycle of length $\ell + 1$ for every $(X, Y)\in \XX\times \YY$.

\subsubsection{Boosters}

Suppose $G\in \SE_2$ has $s_2(G) < n$, and let $\FF_2(G)$ be the set of paths and cycles $F\subseteq G$ with $|F| = s_2(G)$. We first take care of a special case.
\begin{lemma}\label{lem:cycles}
Suppose $G\in \SE_2$ and that $\FF_2(G)$ contains a cycle. Let $\BW_1(G)$ denote the set of edges $e$ such that $s_2(G + e) > s_2(G)$. Then $R_G[\BW_1(G)] = \OM(dn)$.
\end{lemma}
\begin{proof}
The family $\FF_2(G)$ contains a cycle $C$ only if the vertex set $U$ of $C$ forms a connected component in $G$. Since $G$ expands, any connected component has size between $\b n$ and $(1-\b)n$. Since $U\times \ol U \subseteq \BW_1(G)$, by Lemma~\ref{lem:EML} we have $R_G[\BW_1(G)] \ge R(U, \ol U) = \OM(dn)$.
\end{proof}
Now suppose $\FF_2(G)$ contains only paths. For $U\subseteq V$, let $\EP_2(U)$ be the set of ordered pairs $(x, y)$ such that some $P\in \FF_2(G)$ has endpoints $x$ and $y$, and $V(P) = U$.
\begin{lemma}\label{lem:EPk}
Suppose $G\in \SE_2$ with $s_2(G) < n$. Then there exists some $U\subseteq V$ with $|\EP_2(U)| \ge (\b n)^2$.
\end{lemma}

\begin{proof}
P\'osa's lemma (see e.g.~\cite{FriezeKaronski}) shows that if $\{x, y\}\in \EP_2(U)$ for some $x, y$, then the set $Y = \{y : \{x, y\}\in \EP_2(U)\}$ has $|N_G(Y)| < 2|Y|$. Since $G\in \SE_2$ we then have $|Y| \ge \b n$, and the lemma follows just like in Lemma~\ref{lem:IPk}, this time considering ordered pairs.
\end{proof}

We prove the analogue of Proposition~\ref{prop:augmentingwalks} for paths. For a graph $G$, integer $r \ge 1$, and $\th$ tending to infinity, let $\TT_r^\th(G)$ be the family of edge sets $|T| = r$ such that no $e\in T$ is contained in $G$ or incident to $S_\th$, and $s_2(G \cup T) > s_2(G)$.

\begin{proposition}\label{prop:boosterwalk}
Suppose $R\in \RM$ and $G\in \SE_2$ with $s_2(G) < n$. Then there exists an $r\le 2\mix(R) + 1$ such that if $\th$ tends to infinity sufficiently slowly, then
$$
R_G[\TT_r^\th] = \OM(nd^{r}).
$$
\end{proposition}

\begin{proof}
Lemma~\ref{lem:cycles} takes care of the case when $\FF_2(G)$ contains a cycle; it only remains to note that the set $E_\th$ of edges incident to $S_\th$ has $R[E_\th] = o(nd)$ by the degree condition \eqref{eq:powerlaw}. Assume $\FF_2(G)$ contains no cycles.

For a graph $G$ and integer $r\ge 0$, let $\BW_{2r-1}^\th(G)$ be the family of walks $W$ of length $2r-1$ such that (a) $W$ is $P$-alternating for some $P\in \FF_2(G)$, (b) $W$ is strictly $G$-alternating, (c) no vertex appears more than once in $W$, (d) $W$ avoids $S_\th$, and (e) $s_2(P\triangle W) > s_2(P)$. By \eqref{eq:Galt}, it is enough to show that $R_\alt[\BW_{2r-1}^\th(G)] = \OM(nd^r)$ for some $r\le 2\mix(R) + 1$.

Suppose $P$ is a path from $x$ to $y$ on vertex set $U$. Define $\RR_i(P, y)$ as in Section~\ref{sec:rotations}. For $i\ge 0$, let $\BB_{2i+1}(P, y)$ be the set of walks $(w_0,\dots,w_{2i+1})$ with $(w_0,\dots,w_{2i})\in \RR_{2i}$ and $w_{2i+1} \in \{x\} \cup \ol U$, and note that $s_2(P\triangle W) > s_2(P)$ for $W\in \BB_{2i+1}$. 

We define $\rot_{P, y}$ as a random $(R, P)$-alternating walk initated at $y$ with
\al{
  \rot_{P, y}(r_P(Y) \mid Y) & = 1, \quad Y\in \RR_{2i+1}(P, y).
}
This satisfies
\al{
  \rot_{P, y}(\RR_{2i+2}) & = \rot_{P, y}(\RR_{2i+1}), \\
  \rot_{P, y}(\RR_{2i+1}\cup \BB_{2i+1}) & \ge (1-o(1)) \rot_{P, y}(\RR_{2i}).
}
Indeed, if $W = (w_0,\dots,w_{2i+1})\notin \RR_{2i+1}\cup \BB_{2i+1}$ while $(w_0,\dots,w_{2i}) \in \RR_{2i}$, then $\dist_P(w_{2i+1}, \{w_0,\dots,w_{2i}, x\}) < 2$, which happens with probability $o(1)$. We conclude that for any constant $j\ge 1$,
\begin{equation}\label{eq:ABpart}
\rot_{P, y}(\RR_{2j}) + \sum_{i = 0}^{j-1} \rot_{P, y}(\BB_{2i+1}) = 1-o(1).
\end{equation}
For each $(x, y) \in \EP_2(U)$ pick some $P(x, y)\in \FF_2(G)$ on $U$, with $P(x, y)$ and $P(y, x)$ each other's reverses. Define a random $(R, G)$-alternating walk
$$
\rot = \sum_{(x, y)\in \EP_2(U)} \frac{1}{|\EP_2(U)|} \rot_{P(x, y), y}.
$$
Since $|\EP_2(U)| = \OM(n^2)$, this is near-uniform. Let $j = \mix(R)$, and pick $\th \le \lam(R)^{-1/4}$ so that $\th^{4j+1}\|M\| = o(1)$.

{\bf The one-sided case.} Suppose $s_2(G) = n-\OM(n)$. Then $|\ol U| = \OM(n)$, and $\{\t(\ol U) \le 1\} \subseteq \BB_1$. By Lemma~\ref{lem:hitlargeset}, we have
$$
\rot(\BB_{1}\setminus \DD_1^\th) = \OM(1).
$$
Since $\rot$ is near-uniform, \eqref{eq:mixpitoR} shows that $R_\alt[\BB_1\setminus \DD_1^\th] = \OM(nd)$. Since $\BB_1 \setminus \DD_1^\th \subseteq \BW_1^\th(G)$, that finishes this case.

{\bf The two-sided case.} Suppose $s_2(G) = n-o(n)$. Then $\ol U = o(n)$, and Lemma~\ref{lem:hitlargeset} gives $\rot(\t(\ol U\cup S_\th) \le 2j) = o(1)$, so $\rot(\RR_{2j}) = 1-o(1)$. Then there exists a set $\XY\subseteq \EP_2$ such that all $(x, y)\in \XY$ have $\rot_{P(x, y), x}(\RR_{2j}^\th) = 1-o(1)$ and $\rot_{P(x, y), y}(\RR_{2j}^\th) = 1-o(1)$.

Fix $(x, y)\in \XY$ and let $P = P(x, y)$, directed from $x$ to $y$. We will construct a compatible pair $(\XX_{2j}, \YY_{2j})$ of families of walks of length $2j$. Then $\XX_{2j}\circ \YY_{2j} \subseteq \BB_{4j+1}(P, y)$, and we can apply the same techniques that we used for matchings.

{\bf Construction of $\XX_i, \YY_i$.} Initially let $\XX_0 = \{(x)\}$. Let $\t_0$ be a permutation such that $\rot_{P, y}(\t((x), Y) = \t_0) \ge \rot_{P, y}(\RR_{2j}(P, y)) / 2^j$, and let $\YY_0$ be the set of $Y\in \RR_{2j}(P, y)$ with $\t((x), Y) = \t_0$. For $i = 1,\dots,2j$ we inductively construct compatible families $\XX_i, \YY_i$ with $\rot_{P, x}(\XX_i) \ge c_{i, j}$ and $\rot_{P, y}(\YY_i) \ge c_{i, j}$, where $c_{i, j} = (3j)^{-2i}$ for $i = 1,\dots,2j$.

{\bf Construction, odd $i$.} Suppose $(\XX_{i-1}, \YY_{i-1})$ is a compatible pair of families. Define $\XX_i'$ as the set of one-step extensions of $\XX_{i-1}$:
$$
\XX_i' = \{(x_0,\dots,x_{i-1}, v) : (x_0,\dots,x_{i-1})\in \XX_{i-1}, v \in V\}.
$$
We aim to apply Lemma~\ref{lem:observation} to $\t(X, Y)$ for $(X, Y)\in \XX_i' \times \YY_{i-1}$. In order to do this, we need to define a total ordering on $\XX_i'$ such that $X\mapsto \t(X, Y)$ respects some common sequence for all $Y\in \YY_{i-1}$.

Let $\le_\t$ be some arbitrary ordering of the set $\t(\XX_i') = \{\t(X) : X\in \XX_i'\}$. Note that $|\t(\XX_i')| \le i+1$. Define a total ordering $\le$ on $\XX_i'$ such that $X_1 \le X_2$ whenever $\t(X_1) \le_\t \t(X_2)$, or $\t(X_1) = \t(X_2)$ and $f(X_1) \le_P f(X_2)$.

Let $\t' \in \t(\XX_i')$ and consider the set $\XX_i'(\t')$ of $X\in \XX_i'$ with $\t(X) = \t'$. Restricted to this set, our ordering orders walks by their final vertex. Suppose $X_1, X_2\in \XX_i'(\t')$ and $Y\in \YY_{i-1}$ are such that $f(X_1) \le_P f(X_2)$ with no vertex of $Y$ in the interval $[f(X_1), f(X_2)]$ (note that no vertex of $X_1\cup X_2$ is in this interval as then $\t(X_1)\ne \t(X_2)$). Then $\t(X_1, Y) = \t(X_2, Y)$. 

As $X$ runs through $\XX_i'(\t')$ according to the ordering $\le$, the value of $\t(X, Y)$ changes only when $f(X)$ enters or exits $\wh N_P(Y)$, which happens at most $2j + 1$ times. This shows that the map $X\mapsto (X, Y)$, restricted to $\XX_i'(\t')$, respects some sequence of length at most $2j+2$, common to all $Y\in \YY_i$ (see Figure~\ref{fig:tau}). Since $|\t(\XX_i')| \le i+1$, the sequence $X\mapsto \t(X, Y)$ on all of $\XX_i'$ respects some common sequence of length at most $\ell = 2i(j+1) \le (3j)^2$.

We apply Lemma~\ref{lem:observation} with measures $\rot_{P, x}$ and $\rot_{P, y}$ on $\XX_i'$ and $\YY_{i-1}$, respectively. The lemma asserts that there exist sets $\XX_i \subseteq \XX_i'$ and $\YY_i\subseteq \YY_{i-1}$ such that $\t$ is constant on $\XX_i\times \YY_i$, and
$$
\rot_{P, x}(\XX_i) \ge \frac{\rot_{P, x}(\XX_{i-1})}{\ell}\quad \AND\quad \rot_{P, y}(\YY_i) \ge \frac{\rot_{P, y}(\YY_{i-1})}{\ell}.
$$
By induction, both quantities are at least $c_{i-1, j} / \ell \ge c_{i,j}$. Note that if $\t(X, Y) = \perp$ then either $X\in \BB_i(P\triangle Y, x)$, or $\dist_P(f(X), Y) < 2$. The latter has probability $O(\|M\||Y|) = o(1)$. Since $(x, y)\in \XY$, for any $Y\in\YY_{i-1}$ we then have
$$
\rot_{P, x}(\t(\ \cdot\ , Y) = \perp) \le \rot_{P, x}(\t(\ol U) \le i) + o(1) = o(1).
$$
We conclude that the common value of $\t$ on $\XX_i\times \YY_i$ is not $\perp$, so $(\XX_i, \YY_i)$ is compatible.

{\bf Construction, even $i$.} Suppose $\XX_{i-1}, \YY_{i-1}$ have been constructed. Lemma~\ref{lem:orient} shows that $\XX_i = r_{P, \YY_{i-1}}(\XX_{i-1})$ and $\YY_i = \YY_{i-1}$ are compatible. We have $\rot_{P, x}(\XX_i) = \rot_{P, x}(\XX_{i-1})$.

{\bf Gluing.} We proceed exactly as in \eqref{eq:glue2} to obtain
\al{
  R_\alt[\XX_{2j}\circ \YY_{2j}] & = \sum_{u, v} R_\alt[\XX_{2j}^{\to u}] R(u, v) R_\alt[\YY_{2j}^{\to v}] = \OM\bfrac{d^{2j+1}}{n}.
}
Since $\XX_{2j}\circ \YY_{2j}\setminus (\CC_{4j+1} \cup \DD_{4j+1}^\th) \subseteq \BW_{4j+1}^\th(G)$ for any $(x, y)\in \XY$, summing over $(x, y)\in \XY$ and applying Lemmas~\ref{lem:hitlargeset}~(\ref{item:nohit}) and \ref{lem:improper} gives
$$
R_\alt[\BW_{4j+1}^\th(G)] = \OM(nd^{2j+1}) - o(nd^{2j+1}).
$$
\end{proof}

\section{Sprinkling}\label{sec:sprinkling}

Suppose $X$ is a finite set, $p : X\to [0, 1]$ a function, and $\GG = (X, \EE)$ a graph. Let $X_p \subseteq X$ be the random set obtained by independently including any $x\in X$ with probability $p(x)$. Suppose for some $r \ge 1$ that $\TT$ is a set of paths on $r$ vertices in $\GG$. We are interested in the probability that some $T\in \TT$ is contained in $X_p$.

For a set $A\subseteq X$ write $p(A) = \sum_{x\in A} p(x)$ and $p[A] = \prod_{x\in A} p(x)$, and for a family $\AA$ of sets write $p[\AA] = \sum_{A\in \AA} p[A]$. Let $\D$ be the maximum value of $p(N_\GG(x))$ for $x\in X$.
\begin{lemma}\label{lem:sprinkling}
Suppose $r\ge 1$ is an integer, and $\TT$ is a set of paths on $r$ vertices in $\GG = (X, \EE)$, such that $p[\TT] \ge 6(3\D)^{r-2}p(X)$. Then
\al{
  \Prob{T\nsubseteq X_p, \text{ all $T\in \TT$}} \le \exp\left\{-\frac{p[\TT]}{(3\D)^{r-1}r^r}\right\}.
}
\end{lemma}

We begin our proof with the following lemma.
\begin{lemma}\label{lem:PQ}
Suppose $r > 1$. Let $\TT_1$ be the random set of paths $(x_2,\dots,x_r)$ such that $(x_1,\dots,x_r)\in \TT$ for some $x_1\in X_p$. If $p[\TT] \ge 6\D^{r-2}p(X)$, then
\al{
  \Prob{p[\TT_1] < \frac{p[\TT]}{3\D}} \le \exp\left\{-\frac{p[\TT]}{3\D^{r-1}}\right\}.
}
\end{lemma}

\begin{proof}
For $x\in X$ let $\QQ(x)$ be the set of $(x_2,\dots,x_r)$ such that $(x,x_2,\dots,x_r)\in \TT$. For $S\subseteq X$ let $\QQ(S) = \cup_{x\in S} \QQ(x)$. For $S\subseteq X$ and $x\notin S$, say that $x$ is {\em $S$-useful} if
$$
p[\QQ(x) \setminus \QQ(S)] \ge \frac{p[\TT]}{3p(X)}.
$$
\begin{claim}\label{cl:useful}
Suppose $S\subseteq X$ has $p(\QQ(S)) < p[\TT]/3\D$. Then the set $U\subseteq X\setminus S$ of $S$-useful elements has $p(U) \ge p[\TT] / 3\D^{r-1}$. 
\end{claim}

\begin{proof}[Proof of Claim~\ref{cl:useful}]
For $T = (x_2,\dots,x_r)$ let $\QQ^{-1}(T)$ be the set of $x$ such that $(x,x_2,\dots,x_r)\in \TT$. Then $\QQ^{-1}(T) \subseteq N_\GG(x_2)$, and $p(\QQ^{-1}(T)) \le \D$. Let $\MM(S)$ be the set of $(x_1,\dots,x_r)\in \TT$ with $(x_2,\dots,x_r)\notin \QQ(S)$. Then
\begin{equation}\label{eq:MSlo}
p[\MM(S)] \ge p[\TT] - \D p(\QQ(S)) \ge \frac23 p[\TT].
\end{equation}
For any $x\in X$ we have $p[\QQ(x)] \le \D^{r-1}$, so
\al{
  p[\MM(S)] & \le \sum_{x\notin S} p(x) p[\QQ(x)\setminus \QQ(S)] \\
  & \le p(U) \D^{r-1} + p(X\setminus S) \frac13 \frac{p[\TT]}{p(X)} \le p(U)\D^{r-1} + \frac13 p[\TT].\label{eq:MShi}
}
Combining \eqref{eq:MSlo} and \eqref{eq:MShi} gives $p(U) \ge p[\TT]/3\D^{r-1}$.
\end{proof}
Consider sampling $S \subseteq X_p$ by the following procedure.
\begin{enumerate}
\item Initially let $S_0 = \emptyset$ and $Z_0 = \emptyset$. Set $i = 1$.
\item Let $x_i\notin Z_{i-1}$ be an $S_{i-1}$-useful element. Let $Z_i = Z_{i-1}\cup \{x_i\}$, and with probability $p(x_i)$ let $S_i = S_{i-1}\cup\{x_i\}$, otherwise let $S_i = S_{i-1}$.
\item If $p[\QQ(S_i)] \ge p[\TT] / 3\D$, declare \textsc{success} and end the procedure. If $p(Z_i) \ge  p[\TT] / 3\D^{r-1}$, declare \textsc{failure} and end the procedure. Otherwise, increase $i$ by 1 and go to step 2.
\end{enumerate}

By Claim~\ref{cl:useful}, Step 2 can be carried out as long as neither \textsc{success} nor \textsc{failure} has been declared, since then $p(U\setminus Z_{i-1}) > 0$ where $U$ is the set of $S_{i-1}$-useful elements.

Since each $x_i$ is $S_{i-1}$-useful at time of sampling, we have
$$
p[\QQ(S_\ell)] \ge \frac{p[\TT]}{3p(X)} \sum_{i = 1}^\ell \xi_i,
$$
where the $\xi_i$ are independent indicator random variables. Letting $\xi(\ell) = \sum_{i\le \ell} \xi_i$, we have $\E{\xi(\ell)} = p(Z_\ell)$. If \textsc{failure} is declared, there exists some $\ell$ for which $p(Z_\ell) \ge p[\TT] / 3\D^{r-1}$ while $\xi(\ell) < p(X) / \D = o(p(Z_\ell))$. By the Chernoff bound \eqref{eq:chernoff2} we have
\al{
  \Prob{\xi(\ell) < \frac{p(X)}{\D} \AND p(Z_\ell) \ge \frac{p[\TT]}{3\D^{r-1}}} \le \exp\left\{- \frac{p[\TT]}{24\D^{r-1}}\right\}.
}
Since $p[\TT_1] \ge p[\QQ(S_\ell)]$, the lemma follows. 
\end{proof}
We can now prove Lemma~\ref{lem:sprinkling}.
\begin{proof}[Proof of Lemma~\ref{lem:sprinkling}]
If $r = 1$ then $\TT$ is a collection of elements of $X$, and
\begin{equation}\label{eq:ris1}
\Prob{\TT\cap X_p = \emptyset} = \prod_{x\in \TT} (1-p(x)) \le e^{-p[\TT]}.
\end{equation}
Suppose $r > 1$. Let $X_1,\dots,X_r$ be independent random subsets of $X$, each sampling any $x\in X$ with probability $p'(x) = p(x) / r$. Then any $x$ is independently in $X_1\cup\dots \cup X_r$ with probability $1 - (1-p(x)/r)^r \le p(x)$.

Let $\TT_0 = \TT$, and for $0 < i < r$ let $\TT_i$ be the random set of $(x_{i+1},\dots,x_r)$ such that $(x_1,\dots,x_i,x_{i+1},\dots,x_r)\in \TT$ for some $x_j\in X_j$, $1\le j \le i$.

Let $\EE_i$ denote the event that $p'[\TT_i] \ge p'[\TT] / (3\D)^i$. Lemma~\ref{lem:PQ} shows that for $0 < i < r$,
\al{
  \Prob{\ol{\EE_i} \mid \EE_{i-1}} \le\exp\left\{-\frac{1}{3^i\D^{r-1}} p'[\TT]\right\}.
}
We then have
$$
\Prob{\ol{\EE_{r-1}}} \le \sum_{i=1}^{r-1}\Prob{\ol{\EE_i}\mid \EE_{i-1}} \le (r-1) \exp\left\{- \bfrac{1}{3\D}^{r-1} p[\TT] \right\}.
$$
Finally, note that $\TT_{r-1}$ is a set of elements in $X$. Repeating the argument behind \eqref{eq:ris1} gives
$$
\Prob{T\nsubseteq X_p,\text{ all $T\in \TT$} \mid \EE_{r-1}} \le \exp\left\{- \bfrac{1}{3\D}^{r-1} p'[\TT]\right\}.
$$
With $p'[\TT] = p[\TT] / r^r$, this finishes the proof.
\end{proof}

\section{Finishing the high-level proof}

We can now prove Lemma~\ref{lem:thebiglemma}. Suppose $R\in \RM$ and $G\in \SE_i$ with $s_i(G) < \fl{in / 2}$ for some $i\in \{1, 2\}$. Note that $d = \Theta(\ln n)$ since $\g_1(R) = 1$. Let $\th$ tend to infinity arbitrarily slowly, and let $E_\th$ be the set of edges incident to $S_\th = S_\th(R, G)$. Propositions~\ref{prop:augmentingwalks} and \ref{prop:boosterwalk} show, for $i=1, 2$ respectively, that there exists an $r \le 2\mix(R) + 1$ and a set $\TT_r$ of edge sets $T$ with $|T| = r$ and $T\cap (G\cup E_\th) = \emptyset$ such that
$$
R_G[\TT_r] = \sum_{T\in \TT_r} \prod_{uv\in T} R(u, v) = \OM(nd^r).
$$
Suppose $E$ is an edge set with $R(E) = o(n\ln n)$, and that $p$ satisfies $p(u, v) \ge R(u, v) / 3$ for all $\{u, v\}\notin E$. Let $X = \binom{V}{2} \setminus (G \cup E_\th)$. We then have
\al{
  \Prob{s_i(G^p) = s_i(G)} = \Prob{T\nsubseteq X_p, \text{ all $T\in \TT_r$}}.
}
Let $\GG = (X, \EE)$ be the graph on $X$ where $u_1v_1, u_2v_2\in X$ are adjacent if $G$ contains an edge between $\{u_1,v_1\}$ and $\{u_2,v_2\}$. Then
$$
\D = \max_{uv\in X} p(N_\GG(uv)) \le 2\th d.
$$
Let $\TT_r(E)$ be the set of $T\in \TT_r$ which intersect $E$. Picking $\th$ so that $R(E)\th^{r-1} = o(nd)$, we have
$$
R_G[\TT_r(E)] \le R(E)\D^{r-1} = o(nd^r).
$$
It follows that
$$
p[\TT_r] \ge \frac13R_G[\TT_r \setminus \TT_r(E)] = \OM(nd^r).
$$
Note that $p(X) = O(nd)$. Lemma~\ref{lem:sprinkling} then gives
$$
\Prob{T\nsubseteq X_p, \text{ all $T\in \TT_r$}} = \exp\left\{-\OM\bfrac{nd^r}{\D^{r-1}}\right\} = \exp\left\{-\OM\bfrac{nd}{\th^{r-1}}\right\}.
$$
Letting $\th^{r-1} = o(\sqrt{d})$ and recalling that $d = \Theta(\ln n)$ finishes the proof.

\section{Proofs for Section~\ref{sec:altmix}: alternating walks mix}\label{sec:alternatingproofs}

Suppose $G$ is a graph and $R$ a rate matrix on $V$, with associated transition matrix $M$. In Section~\ref{sec:altmix} we defined the simple, lazy $(R, G)$-alternating random walk, which is a special case of the following definition. Recall that $\wh N(A) = A\cup N(A)$.
\begin{definition}
Given a graph $G$, a random walk $\pi$ on $V$ is a {\em random $(R, G)$-alternating walk} if the following hold for all $j \ge 0$:
\al{
  \pi(w_{2j+1} \mid w_0,\dots,w_{2j}) & = M(w_{2j}, w_{2j+1}), \\
  \pi(w_{2j+2} \mid w_0,\dots,w_{2j+1}) & = 0, \quad w_{2j+2}\notin \wh N_G(w_{2j+1}).
}
\end{definition}
In short, a random $(R, G)$-alternating walk alternates between memoryless transitions weighted by $M$, and (lazy) steps restricted to the edges of $G$. We use $\pi_j$ to denote the measure induced by the $j$--th vertex $w_j$, and note that the initial distribution $\pi_0$ may be any distribution on $V$.

Before going into mixing of the random $(R, G)$-alternating walk, we restate and prove Lemma~\ref{lem:improper}. Define $S_{\th}$ as the set of vertices $u$ with $d_G(u) \ge \th$ or $d_R(u) \ge \th d$, and say that a walk avoids $S_\th$ if it contains no vertex of $S_\th$.
\begin{lemma}
Suppose $R\in \RM$ and that $G$ is light-tailed, and let $j = O(1)$. Let $\CC_j^{\th}$ be the set of $G$-alternating walks of length $j$ which avoid $S_{\th}$ and either (a) repeat some vertex, or (b) are not strictly $G$-alternating. If $\th^{j+1}\|M\| = o(1)$ then $R_\alt[\CC_j^{\th}] = o(nd^{\cl{j/2}})$.
\end{lemma}
\begin{proof}
If a $G$-alternating walk $W = (w_0,\dots,w_j)$ repeats a vertex or has $\{w_{2i}, w_{2i+1}\}\in G$ for some $i$, there must exist some $i$ such that $w_{2i+1}\in \wh N_G(w_0,\dots,w_{2i})$. If $d_G(w_i) < \th$ for all $i$, then
$$
M(w_{2i}, \wh N_G(w_0,\dots,w_{2i})) \le \|M\| (2i+1)\th,
$$
so $\pi(\CC_j^\th) = O(\th\|M\|)$ for any random $(R, G)$-alternating walk $\pi$. Note that for any walk $W = (w_0,\dots,w_j)$ which avoids $S_\th$,
\begin{multline}
R_\alt[W] \le \prod_{i=0}^{\cl{j/2}-1} \frac{\th d}{d_R(w_{2i})} R(w_{2i}, w_{2i+1}) \prod_{i=0}^{\fl{j/2} - 1} \frac{\th}{d_G(w_{2i-1})+1}\\
\le \th^jd^{\cl{j/2}} \times n\pi_{G}(W),
\end{multline}
where $\pi_{G}$ is the simple, lazy $(R, G)$-alternating walk initiated at a vertex chosen uniformly at random. Since $\th^{j+1}\|M\| = o(1)$, it follows that $R_\alt[\CC_j^\th] = o(nd^{\cl{j/2}})$.

\end{proof}

\subsection{Mixing for the random alternating walk}

The $R$-steps of the $(R, G)$-alternating walk gets $\pi_j$ closer to stationarity by Lemma~\ref{lem:norm}, while the $G$-steps may pull it back. The following lemma bounds the harm done.
\begin{lemma}\label{lem:semirandom}
Suppose $R\in \RM$ with $\lam(R) = \lam$ and stationary distribution $\s$. Suppose $\pi$ is a random $(R, G)$-alternating walk for some graph $G$ with maximum degree $\th - 1$ and $d_G(u) = 0$ whenever $d_R(u) \ge \th d$. Then for $i\ge 0$,
\al{
  \m_\s(\pi_{2i+1}) & \le \lam \m_\s(\pi_{2i}), \label{eq:semirandomodd}\\
  \m_\s(\pi_{2i+2})^2 & \le \th^4\left(\m_\s(\pi_{2i+1})^2 + 1\right).\label{eq:semirandomeven}
}
In particular, if $\lam\th^2 = o(1)$ then for all $i\ge 0$,
\begin{equation}\label{eq:justaddxi}
\m_\s(\pi_{2i+1})^2 \le \lam^{2i}\th^{4i} \m_\s(\pi_1)^2 + O(\lam^2\th^4).
\end{equation}
\end{lemma}
\begin{proof}
Lemma~\ref{lem:norm} immediately  gives \eqref{eq:semirandomodd}, and we prove \eqref{eq:semirandomeven}. Note that for any $i\ge 1$ and $u\in V$,
\al{
  \pi_{2i}(u) & = \sum_{v\in  \wh N(u)} \pi_{2i-1}(v)\pi(w_{2i}=u\mid w_{2i-1}=v) \le \th \max_{v\in \wh N(u)}\pi_{2i-1}(v).
}
Suppose $v\in \wh N(u)$. Then $u = v$ or $d_R(u) \le \th d$, since $d_G(u) = 0$ whenever $d_R(u) \ge \th d$. Since $d_R(v) \ge d$, in either case we conclude that $\s(u) / \s(v)  = d_R(u) / d_R(v) \le \th$, and
\al{
  \m_\s(\pi_{2i})^2  & \le \sum_v \frac{\pi_{2i}(v)^2}{\s(v)}\le \sum_v \frac{\th^2}{\s(v)} \max_{u\in \wh N(v)}\pi_{2i-1}(u)^2.
}
Any vertex $u$ is counted at most $\th$ times in this sum, so
\al{
  \m_\s(\pi_{2i})^2 \le \th^3\sum_u \frac{\pi_{2i-1}(u)^2}{\s(u)}\max_{v\in \wh N(u)}\frac{\s(u)}{\s(v)} \le \th^4 \left(\m_\s(\pi_{2i-1})^2 + 1\right). \label{eq:oddtoeven}
}
This shows \eqref{eq:semirandomeven}. Repeatedly applying \eqref{eq:semirandomodd} and \eqref{eq:semirandomeven} with $\lam\th^2 = o(1)$ gives \eqref{eq:justaddxi}.
\end{proof}

Recall that for a family $\WW$ of walks and a vertex $v$, $\WW^{\to v}$ is the set of walks in $\WW$ ending at $v$. Say that a walk $W = (w_0,\dots,w_j)$ is {\em non-lazy} if $w_{i} \ne w_{i+1}$ for all $0\le i < j$.

For any random $(R, G)$-alternating walk $\pi$ we define a variant $\pi^\th$ by the following holding for any $W = (w_0,\dots,w_{2j-1})$: if $w_{2j-1}\in S_\th$ then $w_{2j} = w_{2j-1}$, and if $w_{2j-1}\notin S_\th$ then
\begin{equation}
\pi^\th(w_{2j}\mid W) = \left\{\begin{array}{ll}
0, & w_{2j}\in S_\th, \\
\pi(w_{2j}\mid W) + \sum_{v\in S_\th} \pi(v\mid W), & w_{2j} = w_{2j-1}, \\
\pi(w_{2j}\mid W), & w_{2j}\notin \{w_{2j-1}\}\cup S_\th.
\end{array}\right.\label{eq:pith}
\end{equation}
Then $\pi^\th$ is a random $(R, G^\th)$-alternating walk, where $G^\th\subseteq G$ is obtained from $G$ by removing any edge incident to $S_\th$. The walk $\pi^\th$ is designed to satisfy the conditions of Lemma~\ref{lem:semirandom} as well as satisfying $\pi^\th(\WW) = \pi(\WW)$ for any family $\WW$ of walks which either (a) are non-lazy and avoid $S_\th$, or (b) have $w_i\notin \wh N(S_\th)$ for odd $i$.

\begin{lemma}\label{lem:RGmix}
Suppose $R \in \RM$ with $\lam(R) = \lam$, suppose $F$ is a graph with maximum degree $\D(F)\le 2$, and suppose $\pi$ is a random $(R, F)$-alternating walk. Suppose $\th \le \lam^{-1/4}$ tends to infinity with $n$, and let $j\ge 2-2\frac{\ln (n\|M\|)}{\ln\lam}$ be an integer. Let $c > 0$ be constant. Suppose $\WW$ is a set of non-lazy $S_\th(R, F)$-avoiding walks of length $2j$, such that $\pi(\WW) \ge c$. Then there exists a constant $\r > 0$ such that there exists a set $|U_{2j}|\ge \r n$, such that any $u\in U_{2j}$ has $\pi(\WW^{\to u}) \ge \r/ n$.
\end{lemma}

\begin{proof}
Let $S_\th = S_{\th}(R, F)$ and consider the walk $\pi^\th$ defined above. Since $\th \le \lam^{-1/4}$ and $\lam=o(1)$, Lemma~\ref{lem:semirandom} implies that
\begin{equation}\label{eq:downtopi1}
\m_\s(\pi_{2j-1}^\th)^2 \le \lam^{j-1}\m_\s(\pi_1^\th)^2 + O(\lam).
\end{equation}
We have $\pi_1^\th(u) = \sum_v \pi_0^\th(v)M(v, u) \le \|M\|$ for all $u$, and since $\s(u) \ge 1/bn$ for all $u$,
$$
\m_\s(\pi_1)^2 \le \sum_u \frac{\pi_1(u)^2}{\s(u)} \le bn^2\|M\|^2.
$$
So for $j\ge 2 - 2\ln(n\|M\|)/\ln(\lam)$, \eqref{eq:downtopi1} becomes $\m_\s(\pi_{2j-1}^\th)^2 = O(\lam)$. Define $T$ as the set of vertices $u$ with $|\pi_{2j-1}^\th(u) - \s(u)| < \lam^{1/3}\s(u)$. Then by definition of $\m_\s$,
\al{
  O(\lam) = \m_\s(\pi_{2j-1}^\th)^2 \ge \sum_{v\notin T} \left(\frac{\pi_{2j-1}^\th(v)}{\s(v)} - 1\right)^2\s(v) \ge \lam^{2/3}\s(\ol T).
}
Then $\s(T) = 1-O(\lam^{1/3})$. By definition of $T$,
\al{
  \pi_{2j-1}^\th(T) & \ge (1-\lam^{1/3})\s(T) = 1-O(\lam^{1/3}), \label{eq:pi2jT}
}
and for any vertex set $A$,
\begin{equation}\label{eq:Abound}
\pi_{2j-1}^\th(A) \le (1+\lam^{1/3}) \s(A\cap T) + \pi_{2j-1}^\th(\ol T) \le \s(A) + O(\lam^{1/3}).
\end{equation}

Let $\WW_{2j-1}\subseteq \WW$ be the set of walks obtained by removing the final vertex from walks in $\WW$. Note that $c \le \pi(\WW_{2j-1}) = \pi^\th(\WW_{2j-1})$. For any $v\in V$ let $\WW_{2j-1}^{\to v}$ be the set of walks in $\WW_{2j-1}$ which end at $v$, and let $U_{2j-1}$ be the set of $v$ such that $\pi^\th(\WW_{2j-1}^{\to v}) \ge c/2n$. Then
\begin{multline}
c \le \pi^\th(\WW_{2j-1})
= \sum_v \pi^\th(\WW_{2j-1}^{\to v}) \\
\le \frac{c}{2n}|\ol{U_{2j-1}}| + \pi_{2j-1}^\th(U_{2j-1})
\le \frac{c}{2} + \s(U_{2j-1}) + O(\lam^{1/3}).
\end{multline}
Since $\lam = o(1)$, we conclude that $\s(U_{2j-1}) \ge c/3$. By Lemma~\ref{lem:sublinear} we then have $|U_{2j-1}| \ge c' n$ for some constant $c' > 0$. Let
$$
U_{2j} = \left\{v : \pi^\th(\WW^{\to v}) \ge \frac{1}{3}\frac{c}{2n}\right\}.
$$
Since $\D_F\le 2$, every $u\in U_{2j-1}$ has at least one neighbour in $U_{2j}$, counting self-loops, and $|U_{2j}| \ge |U_{2j-1}| / 3 \ge c' n / 3$. Since $\pi(\WW^{\to v}) = \pi^\th(\WW^{\to v})$ for all $v$, this finishes the proof with $\r = \min\{c'/3, c/6\}$.
\end{proof}

\subsection{Hitting times for sets}

For a random walk $(w_0,w_1,\dots)$ and $S\subseteq V$, recall that $\t(S)$ is the smallest $i$ for which $w_i\in S$.
\begin{lemma}\label{lem:badwalks}
Suppose $R\in \RM$ and that $G$ is light-tailed. Suppose $\pi$ is a random $(R, G)$-alternating walk with $\pi_0(u) = 1/|A|$ for $u\in A$, where $A\subseteq V$ has size $|A| = \OM(n)$.
\begin{enumerate}[(i)]
\item If $j\ge 0$ is constant and $|S| = o(n)$, then
$$
\pi(\t(S) \le 2j) = o(1).
$$
\item If $|S| = \OM(n)$ then
$$
\pi(\t(S) \le 1) = \OM(1).
$$
\end{enumerate}
\end{lemma}

\begin{proof}
We prove (i). Let $\th = o(n/|S|)$ tend to infinity with $n$. We may assume that $S_\th\subseteq S$, since replacing $S$ by $S\cup S_\th$ only increases the probability in question. Note that since $G$ is light-tailed,
$$
|\wh N(S)| \le |\wh N(S\setminus S_\th)| + |\wh N(S_\th)| \le \th|S| + o(n) = o(n).
$$
Note that either $\t(S) = 0$ or $\t(S) \ge \t_\odd(\wh N(S))$. We then have
\begin{equation}\label{eq:piz}
\pi(\t(S) \le 2j) \le \pi_0(S) + \pi(\t_\odd(\wh N(S)) < 2j \mid \ol S).
\end{equation}
The first term equals $|S| / |A| = o(1)$.

Let $\pi^\th$ be the modification of $\pi$ defined in \eqref{eq:pith}. Note that since $S_\th \subseteq S$, any walk $W\in \{\t_\odd(\wh N(S)) < 2j\}$ has $\pi^\th(W) = \pi(W)$. Since $\s(u) \ge 1/bn$ for all $u$ for some constant $b\ge 1$,
\al{
  \m_\s(\pi_0^\th)^2 & = \left(\sum_{u\in Z} \frac{(1/|Z|)^2}{\s(u)}\right) - 1 \le \frac{bn}{|Z|} - 1 = O(1).\label{eq:pi0lin}
}
By Lemma~\ref{lem:semirandom} and since $\th\le\lam^{-1/4}$, for any odd $i \ge 1$ we have $\m_\s(\pi_{i}^\th)^2 \le \m_\s(\pi_1^\th)^2 + O(\lam)  = O(\lam)$.
As in \eqref{eq:Abound}, we have $\pi_i^\th(\wh N(S)) \le \s(\wh N(S)) + o(1)$. Then
$$
\pi(\t(S)\le 2j) \le o(1) + \pi^\th(\t_\odd(\wh N(S)) < 2j) \le \sum_{\substack{i < 2j \\ i \text{ odd}}} \pi_i^\th(\wh N(S)) = o(1).
$$

Part (ii) follows from $\m_\s(\pi_1)^2 = O(\lam)$. Applying \eqref{eq:Abound} to the complement of $S$, we have
$$
\pi(\t(S) \le 1) \ge \pi_1(S) \ge \s(S) - o(1).
$$

\end{proof}

\section{A low-degree expander}\label{sec:Hprops}

Recall that $G_{n, R}(t)$ is constructed by letting $E(u, v)$ be independent exponential random variables with rate $R(u, v)$ for all $\{u, v\}$, including any edge $\{u, v\}$ with $E(u, v) \le t$ (note that $E(u, v) = E(v, u)$). Let $D\ge k$ be an integer and define
$$
T_D(u) = \inf\{t > 0 : |\{v : E(u, v) \le t\}| \ge D\}
$$
be the random time at which $u$ attains degree $D$. We define a graph $H(t) \subseteq G_{n, R}(t)$ by including an edge $\{u, v\}$ whenever $E(u, v)\le t$ and $E(u, v) \le \max\{T_D(u), T_D(v)\}$, and let $H = H(\t_k)$.

\begin{lemma}\label{lem:Hprops}
Suppose $R \in \RM$ and $k\ge 1$. There exists some $D = O(1)$ such that the following hold.
\begin{enumerate}[(i)]
\item Let $\th$ tend to infinity with $n$. Letting $S_{\th}$ denote the set of $u$ with $d_H(u) \ge \th$ or $d_R(u) \ge \th d$, with high probability $|\wh N_H(S_{\th})| = o(n)$.
\item There exists a constant $\b = \b(R, k) > 0$ such that with high probability, every $|A| < \b n$ has $|N_{H}(A)| \ge k|A|$.
\end{enumerate}
\end{lemma}

We prove Lemma~\ref{lem:Hprops} over the next few sections. 

\subsection{$H$ and the $D$-out graph}\label{sec:Dout}

For all ordered pairs $(u, v)$, let $X(u, v)$ be independent exponential random variables with rate $R(u, v) / 2$. Define
\al{
  T_D^+(u) & = \inf\{t > 0 : |\{v : X(u, v) \le t\}| \ge D\}, \\
  T_D^-(v) & = \inf\{t > 0 : |\{u : X(u, v) \le t\}| \ge D\}.
}
Define two undirected graphs on $V$ by
\al{
  H_D^+ & = \{\{u, v\} : X(u, v) \le T_D^+(u)\}, \\
  H_D^- & = \{\{u, v\} : X(u, v) \le T_D^-(v)\}.
}
Then $H_D^+$ and $H_D^-$ are equal in distribution, the common distribution being the $R$-weighted $D$-out random graph $\GG_{R, D}$, defined as follows. Each $u$ independently samples $D$ vertices $N(u)$ chosen without replacement with probability proportional to $R(u, \cdot)$. Let $\vec{\GG}_{R, D}$ be the graph with edges $(u, v)$ for $v\in N(u)$. Then $\GG_{R, D}$ is obtained by ignoring orientations and merging parallel edges in $\vec{\GG}_{R, D}$.

We couple $H_D^+$ and $H_D^-$ to $H$ by letting $E(u, v) = \min\{X(u, v), X(v, u)\}$. Then $H\subseteq H_D^+\cup H_D^-$. Indeed, suppose $\{u, v\}\in H$. If $X(u, v) \le X(v, u)$ then
\al{
  X(u, v) & = E(u, v) \le \max\{T_D(u), T_D(v)\} \le \max\{X_D^+(u), X_D^-(v)\},
}
so $\{u, v\}\in H_D^+\cup H_D^-$. The same argument with the signs reversed holds if $X(v, u) \le X(u, v)$.

\subsection{Lemma~\ref{lem:Hprops} (i): degrees in $H$}

We have $S_\th = A_\th \cup B_\th$ where $A_\th = \{v : d_H(v) \ge \th\}$ and $B_\th = \{u : d_R(u) \ge \th d\}$, and
\al{
  |\wh N(S_{\th})| & \le |\wh N(B_\th\setminus A_\th)| + |\wh N(A_{\th})| \le  \th |B_\th| + |\wh N(A_\th)|.
}
Letting $c = 1/20D$, we bound
\begin{multline}
|\wh N_H(A_\th)| \le \sum_{\ell \ge \th} (\ell+1)|\{v : d_H(v) = \ell\}| = \th|A_\th| + \sum_{\ell \ge \th} |A_\ell| \\
\le \th|A_\th \setminus B_{c\th}| + \th|B_{c\th}| + \sum_{\ell \ge \th}(|A_\ell\setminus B_{c\ell}| + |B_{c\ell}|).
\end{multline}
We bound $|A_\ell \setminus B_{c\ell}|$. The discussion in Section~\ref{sec:Dout} shows that $H\subseteq H_D^+\cup H_D^-$ where $H_D^+ \stackrel{d}{=} H_D^- \stackrel{d}{=} \GG_{R, D}$. Letting $d(u)$ denote degrees in $\GG_{R, D}$,
\al{
  \Prob{d_H(u) \ge \ell} \le 2\Prob{d(u) \ge \ell/2}.
}
Let $X_u$ be the number of vertices $v$ with $u\in N(v)$. If $d(u) \ge \ell/2$ then $X_u \ge \ell/2-D \ge \ell/4$. Vertices $v\ne u$ independently have $u\in N(v)$ with probability at most $DM(v, u)$ by Lemma~\ref{lem:exprank}. Since $M(v, u) = d_R(u)M(u, v)/d_R(v)$, and $d_R(u)/d_R(v) \le \ell/20D$ for $u\notin B_{c\ell}$, we have
$$
\E{X_u} = \sum_{v} \Prob{u\in N(v)} \le D\sum_v M(v, u) \le \frac{\ell}{20}.
$$
By the Chernoff bound~\eqref{eq:chernoff4}, we have
\begin{equation}\label{eq:dHbound}
\Prob{d_H(u) \ge \ell} \le 2\Prob{X_u \ge \frac{\ell}{4}} \le 2\bfrac{1}{5}^{\ell/2},\quad u\notin B_{c\ell}.
\end{equation}
It follows that $\mathbb{E}|A_\ell\setminus B_{c\ell}| = ne^{-\OM(\ell)}$. Since $R\in \RM$, there are constants $b, b_1 > 0$ and $\a\in [0, 1/2)$ such that
\begin{equation}\label{eq:detpowlaw}
\frac{t|B_t|}{b_1n} \le \s(B_t) \le b\bfrac{|B_t|}{n}^{1-2\a}.
\end{equation}
If $\a > 0$ then $|B_t| \le b_2 t^{-\frac{1}{1-\a}} n$ for some $b_2 > 0$. Then 
\al{
  \mathbb{E}|\wh N(S_\th)| & \le \th|B_\th| + \th|B_{c\th}| + \th\mathbb{E}|A_\th| + \sum_{\ell \ge \th} |B_{c\ell}| + \mathbb{E}|A_\ell\setminus B_{c\ell}| \\
  & \le \left(\frac{b_2}{\th^{\frac{\a}{1-\a}}} + \frac{b_2}{(c\th)^{\frac{\a}{1-\a}}} + \th e^{-\OM(\th)} + \sum_{\ell\ge\th}\frac{b_2}{(c\ell)^{\frac{1}{1-\a}}} + e^{-\OM(\ell)} \right)n.
}
For $\th$ tending to infinity, Markov's inequalty shows that $|\wh N(S_\th)| = o(n)$ whp. If $\a = 0$ then~\eqref{eq:detpowlaw} gives $|B_t| = 0$ for any $t$ tending to infinity, and we again conclude that $|\wh N(S_\th)| = o(n)$ whp.

\subsection{Lemma~\ref{lem:Hprops} (ii): expansion in $H$}

Note that the distribution of $H$ is unaffected by scaling $R$, and we may assume that $R$ is scaled so that $\g_1(R) = 1$, and in particular $1-\e \le \t_k \le 1+\e$ whp for any $\e \gg \frac{\ln\ln n}{\ln n}$, by Lemma~\ref{lem:threshold}. Let $\SML$ be the set of vertices $u$ with degree less than $D$ in $G_{n, R}(1-\e)$. For $A\subseteq V$, let $A_1 = A\cap \SML$ and $A_2 = A\setminus \SML$, and note that 
\al{
  |N_{H}(A)| & = |N_H(A_1)\setminus A_2| + |N_H(A_2) \setminus \wh N_H(A_1)| \\
  & \ge |N_H(A_1)| + |N_{H}(A_2)| - |A_2| - e_H(A_2, \wh N(\SML)),
}
We proceed in three parts. Firstly, we show that whp no vertex in $H$ has two neighbours in $\SML$, and that $\SML$ contains no edges, which implies that $|N_H(A_1)| \ge k|A_1|$ since $H$ has minimum degree at least $k$. Secondly, we note that $N_H(A_2) = N_{H(\infty)}(A_2)$ since $A_2\cap\SML = \emptyset$, and show that whp $|N_{H(\infty)}(A)| \ge \frac{D}{16}|A|$ for all $|A| \le \b n$, if $D$ is large enough. Lastly, we show that whp $e_H(u, \wh N(\SML)) \le \frac{D}{17}$ for all $u$, if $D$ is large enough. We conclude that if $D$ is large enough then whp, for all $|A| \le \b n$,
$$
|N_H(A)| \ge k|A_1| + \frac{D}{16}|A_2| - |A_2| - \frac{D}{17}|A_2| \ge k|A|.
$$

\subsubsection{Part 1}

Let $t = 1-\e$. For an edge set $F$, let $\SML_F(t)\subseteq \SML(t)$ be the set of vertices $u$ with degree less than $D$, not counting the edges in $F$. Letting $\TT$ denote the event that $t \le \t_k\le 2$, we have
\al{
  \{u, v\in \SML(t)\}\cap \{F\subseteq H\}\cap \TT \subseteq \{u,v\in \SML_F(t)\} \cap \{F\subseteq G_{n, R}(2)\}.
}
The two events in the right-hand side are independent, and we first use Lemma~\ref{lem:degreebound} to bound
\al{
  \Prob{e_H(\SML(t)) > 0} 
  & \le \Prob{\ol \TT} + \sum_{u, v} \Prob{u, v\in \SML_{\{uv\}}(t)}\Prob{E(u, v) \le 2} \\
  & \le o(1) + 2\|R\|\sum_{u, v}p_\e(u)p_\e(v), \label{eq:ss1}
}
for some $p_\e(u) = e^{-(1-\e)d_R(u) + O(\ln d_R(u))}$. Likewise, the probability that some $w$ has two neighbours in $\SML$ is bounded by
\al{
  4\sum_{u,v,w}p_\e(u)p_\e(v) R(u, w)R(v, w) & \le 4\|R\|\sum_{u, v} p_\e(u)p_\e(v)d_R(v) \\
  & \le \|R\|\sum_{u, v}p_\e(u)p_\e(v), \label{eq:ss2}
}
where $4d_R(v)$ is absorbed into the error term of $p_\e(v)$. We bound $\sum_u p_\e(u)$. Recall that $\g_1(R) = \sum_u e^{-d_R(u)} = 1$ by choice of scaling. Let $U$ be the set of $u$ with $d_R(u) \le 2\ln n$. Then, as $p_\e(u) \le e^{-(1-2\e)d_R(u)}$,
\al{
  \sum_u p_\e(u) & \le e^{4\e\ln n}\sum_{u\in U} e^{-d_R(u)} + |\ol U|e^{-(1-2\e)2\ln n} \le 2n^{4\e}.
}
We have $\|R\|n^{8\e} = o(1)$, and conclude that both \eqref{eq:ss1} and \eqref{eq:ss2} are $o(1)$.

\subsubsection{Part 2}

Let $m\ge 1$ be some integer to be chosen. Consider the $2m$-out graph $\GG_{R, 2m}$. Let $N(u)$ be the $2m$ vertices chosen by $u$, independent for all $u$. Fix some set $A\subseteq V$ with $|A| = a \le \b n$, let $\k = ((m+1)a/n)^c$ with $c > 0$ as in Lemma~\ref{lem:EML}~(\ref{item:kappa}). Note that $\k$ can be made smaller than any positive constant by letting $\b$ be small enough, and we choose $\b$ sufficiently small to allow the Chernoff bounds below.

Consider the following procedure. Initially set $B_0 = A$. For $1 \le i \le 3a/4$ do as follows. Let $u_i\in A \setminus\{u_1,\dots,u_{i-1}\}$ be such that $M(u, B_{i-1}) < \k$. Note that this is possible for $i\le 3a/4$ by Lemma~\ref{lem:EML}~(\ref{item:kappa}) since $A\subseteq B_{i-1}$ and $|B_{i-1}| \le (m+1)a$. Reveal vertices of $N(u_i)$ until (a) at least $m$ vertices not in $B_{i-1}$ have been found, in which case we add those $m$ vertices to $B_{i-1}$ to form $B_i$, or (b) all of $N(u_i)$ has been revealed. Let $X_i = 1$ if (a) occurs and 0 otherwise.

When a vertex of $N(u_i)$ is revealed it has probability at most $2\k$ of being in $B_i$ (with the factor 2 accounting for the choices already made). So, conditional on the procedure so far, the probability that $X_i = 0$ is at most
\al{
  \Prob{\Bin{2m}{2\k} \ge m} \le \bfrac{4\k m}{m}^{m/2},
}
by the Chernoff bound \eqref{eq:chernoff4}. With $p = (4\k)^{m/2}$ we then have
\al{
  \Prob{|N(A)| < \frac{m}{2}|A|} \le \Prob{\sum_{i=1}^{3a/4} X_i < \frac{a}{2}} \le \Prob{\Bin{\frac{3a}{4}}{p} \ge \frac{a}{4}}.
}
Again applying \eqref{eq:chernoff4}, we obtain
\al{
  \Prob{\exists |A| \le \b n : |N(A)| < \frac{m}{2}|A|} & \le \sum_{a\le \b n} \binom{n}{a} \bfrac{3ap/4}{a/4}^{a/8} \\
  & \le \sum_{a\le \b n} \left(\frac{ne}{a} (3p)^{1/8}\right)^a.
}
We have $p^{1/8} = f(m)(a/n)^{cm/16}$ for some function $f(m)$. Choosing $m > 16/c$, and $\b$ small enough, we conclude that this sum is $o(1)$. So $\GG_{R, 2m} \in \EE_{m/2}$ whp, where
$$
\EE_{m/2} = \left\{|N(A)| \ge \frac{m}{2}|A| \text{ for all $|A| \le \b n$}\right\}.
$$

Let $D = 4m$. Condition on the whp events $H_{D/2}^+\in\MM_{m/2}$ and $H_{D/2}^-\in \MM_{m/2}$, and let $|A| \le \b n$. Note that for each $u$, $N_{H(\infty)}(u)$ contains at least one of $N^+(u)$ and $N^-(u)$. Let $A^+$ be the set of $u$ with $N^+(u)\subseteq N_{H(\infty)}(u)$, and suppose without loss of generality that $|A^+| \ge |A|/2$. Then
\al{
  |N_{H(\infty)}(A)| & \ge |N^+(A^+)| \ge \frac{m}{2}|A^+| \ge \frac{D}{16}|A|.
}
For $D$ large enough, we conclude that $H(\infty) \in \EE_{D/16}$ whp.

\subsubsection{Part 3}

Let $t = 1+\e$. Fix $u\in V$ and let $\SML_u\supseteq \SML$ be the set of vertices $v$ with degree less than $D$ in $G_{n, R}(t)$, not counting edges incident to $u$. Let $X(u) = |N_t(u)\cap \wh N_t(\SML_u)|$. If $\t_k\le t$ we then have $e_H(u, \wh N_H(\SML)) \le X(u)$ since $H\subseteq G_{n, R}(t)$. Note that $N_t(u)$ and $\wh N_t(\SML_u)$ are independent. Conditional on $E(v, w)$ for all $v,w\ne u$, the expected value of $X(u)$ is
$$
tR(u, \wh N_t(\SML_u)) \le 2\|R\| |N_t(\SML_u)| \le 2D \|R\||\SML_u|.
$$
Let $\f$ be such that $\|R\|^{\f/2} \le n^{-2}$. The Chernoff bound \eqref{eq:chernoff4} then gives
\al{
  \Prob{X(u) \ge \f\ \middle|\ |\SML_u| = n^{o(1)} } & \le \bfrac{\E{X}}{\f}^{\f/2} \\
  & = O\left((\|R\||N_t(\SML_u)|)^{\f/2}\right) = o(n^{-1}).
}
Since $\t_k \le 1+\e$ whp by Lemma~\ref{lem:threshold}, we conclude that $e_H(u, \wh N_H(\SML)) < \f$ for all $u$ whp.

\section{Proofs for Section~\ref{sec:prelims}}\label{sec:prelimproofs}

\subsection{Degrees}

\begin{proof}[Proof of Lemma~\ref{lem:degreebound}]
Suppose $X = X_1 + \dots + X_n$ where the $X_i$ are independent indicator random variables with $\E{X_i} = 1-e^{-\m_i}$ and $\m = \m_1+\dots+\m_n$, where $\m_i /\m \le \e = o(1)$ for all $i$, and $\m$ tends to infinity with $n$. It is not hard to show that
\al{
  \Prob{X \le \ell} & = \frac{e^{-\m}\m^\ell}{\ell!}\left(1 + O\bfrac{\e}{\m}\right),\label{eq:Pxl}
}
Define $\Po(\m, \ell) = e^{-\m}\m^\ell/\ell!$.

Let $R\in \RM(1)$ and let $t = \OM(1)$. For any vertex $u$ and any set $S\subseteq V$ with $|V\setminus S| = O(1)$, $e_t(u, S)$ satisfies the above with
$$
\E{e_t(u, S)} = tR(u, S) = td_R(u) - o(1).
$$
With $d_R(u)\ge d$ tending to infinity we then have
\al{
  \Prob{e_t(u, S) \le \ell} = (1+o(1))\Po(td_R(u), \ell) = e^{-td_R(u) + O(\ln d_R(u))}.
}
\end{proof}

\begin{proof}[Proof of Lemma~\ref{lem:threshold}]
Let $P\in \RM$ be a matrix with $\g_k(P) \to \g_k\in (0, \infty)$. Note that $d = \min d_P(u) = \Theta(\ln n)$ tends to infinity with $n$. Let $U$ be a set of $\ell$ distinct vertices, let $k > 0$, and let $0 < k_u \le k$ for each $u\in U$, with $\sum_{u\in U}(k-k_u) = 2m$ for some $m\ge 0$. Consider the graph $G_{n, P}$. By~\eqref{eq:Pxl},
\al{
  \Prob{e(u, \ol U) < k_u, \text{ all $u\in U$}} & = \prod_{u\in U} \Prob{e(u, \ol U) < k_u} \\
  & = (1+o(1)) \prod_{u\in U} \Po(d_P(u), k_u-1) \label{eq:exactdegree} \\
  & \le \frac{1+o(1)}{d^{2m}} \prod_{u\in U} \Po(d_P(u), k-1). \label{eq:approxdegree}
}
Let $\EE_m$ be the event that $U$ contains exactly $m$ edges. Then $\EE_m$ is independent of $\{e(u, \ol U) : u\in U\}$, and $\Prob{\EE_m} = O(\|P\|^m)$ and $\Prob{\EE_0} = 1-o(1)$. For $k > 0$ we then have, using both \eqref{eq:exactdegree} and \eqref{eq:approxdegree},
\al{
  \Prob{d(u) < k\ \forall u\in U} & = \sum_m\Prob{\EE_m}\sum_{\sum k_u = k\ell-2m} \Prob{e(u, \ol U) < k_u \forall u\in U} \\
  & \le \left(\prod_{u\in U} \Po(d_P(u), k-1)\right)\left(\Prob{\EE_0} + \sum_{m > 0} O\bfrac{\|P\|^m}{d^{2m}}\right) \\
  & = (1+o(1))\prod_{u\in U}\Po(d_P(u), k-1).
}
Letting $X_k$ denote the number of vertices in $G_{n, P}$ with $d(u) < k$,
\al{
  \E{\binom{X_k}{\ell}} & = \sum_{|U| = \ell} \Prob{d(u) < k, \text{ all $u\in U$}} \\
  & = \sum_{|U| = \ell}(1+o(1))\prod_{u\in U} \Po(d_P(u), k-1) = (1+o(1)) \frac{\g_k(P)^\ell}{\ell!}.
}
If $\g_k(P)$ converges to some $\g_k < \infty$, the method of moments (see e.g.~\cite{Durrett10}) implies that $X_k$ converges to a Poisson random variable with expected value $\g_k$, and 
$$
\lim_{n\to \infty} \Prob{\d(G_{n, P}) \ge k} = \lim_{n\to \infty} \Prob{X_k = 0} = e^{-\g_k}. 
$$
If $\g_k(P)$ diverges to infinity, we note that $\Var{X_k} = o(\E{X_k})$, and Chebyshev's inequality implies that $\Prob{X_k > 0} \to 1$.

To obtain a bound for $\t_k$, let $\e \gg \frac{\ln\ln n}{\ln n}$ and suppose $\g_1(R) = 1$. Note that $G_{n, R}(1+\e) \stackrel{d}{=} G_{n, P}$ with $P(u, v) = 1 - e^{-(1+\e)R(u, v)}$. This matrix has $d_P(u) + O(\ln d_P(u)) \ge (1+\e/2)d_R(u)$ for all $u$, where we use the fact that $\|R\| = o(\e)$. We have
\al{
  \g_k(P) = \sum_u e^{-d_P(u)} d_P(u)^{k-1} & \le \sum_u e^{-(1+\e/2)d_R(u)} \\
  & \le e^{-\e d/2} \g_1(R) = o(1).
}
We conclude that $\d(G_{n, R}(1+\e)) \ge k$ whp. By the same token, $\d(G_{n, R}(1-\e)) < k$ whp.

\end{proof}

\subsection{A matrix lemma}

\begin{proof}[Proof of Lemma~\ref{lem:observation}]
If $\ell = 1$, take $S = I$ and $T = J$. We prove the case $\ell > 1$ by induction. By rescaling, we may assume that $\pi_I(I) = 1$ and $\pi_J(J) = 1$.

For each $j\in J$ let $I_\ell(j)$ be the set of $i\in I$ with $\t(i, j) = a_\ell$. Let
$$
J' = \left\{j\in J : \pi_I(I_\ell(j)) < \frac{1}{\ell}\right\}.
$$
If $\pi_J(J') < 1-\ell^{-1}$, let $S = \cap_{j\notin J'}I_\ell(j)$ and $T = J\setminus J'$. Then $\t = a_\ell$ on $S\times T$ and $\pi_I(S) \ge \ell^{-1}$, $\pi_J(T) \ge \ell^{-1}$.

If $\pi_J(J') \ge 1-\ell^{-1}$, let $I' = \cap_{j\in J'} (I\setminus I_\ell(j))$ and consider the matrix
$$
\t'(i, j) = \t(i, j), \quad i \in I', j\in J'.
$$
This takes values $\{a_1,\dots,a_{\ell-1}\}$, and by induction there exist $S\subseteq I', T\subseteq J'$ with $\pi_I(S) \ge (\ell-1)^{-1}\pi_I(I')$ and $\pi_J(T) \ge (\ell-1)^{-1}\pi_J(J')$ such that $\t'$, and therefore $\t$, is constant on $S\times T$. We have
$$
\pi_I(I') \ge \min_{j\in J'} \pi_I(I\setminus I_\ell(j)) \ge 1-\ell^{-1}, \quad \pi_J(J') \ge  1-\ell^{-1},
$$
so
\al{
  \pi_I(S) & \ge (\ell-1)^{-1}\pi_I^j(I') \ge \ell^{-1}, \\
  \pi_J(T) & \ge (\ell-1)^{-1}\pi_J(J') \ge \ell^{-1}.
}

\end{proof}

\subsection{Mixing in simple random walks}

\begin{proof}[Proof of Lemma~\ref{lem:norm}]
This proof is more or less taken from~\cite{LevinPeres17}, with slight modifications. We first note that for $R\in \RM$, the transition matrix $M$ is reversible:
$$
\s(u)M(u, v) = \frac{d_R(u)}{d_R(V)} \frac{R(u, v)}{d_R(u)} = \frac{d_R(v)}{d_R(V)} \frac{R(v, u)}{d_R(v)} = \s(v)M(v, u).
$$

For vectors $f,g : V\to \R$ we define an inner product
\begin{equation}\label{eq:normdef}
\langle f, g\rangle_\s = \sum_v f(v)g(v)\s(v),
\end{equation}
and the associated norm $\|f\|_\s = \langle f, f\rangle_\s^{1/2}$. Then for probability measures $\pi$,
\begin{equation}\label{eq:mudef}
\left\|\frac{\pi(\cdot)}{\s(\cdot)} - \ind\right\|_\s = \sqrt{\left(\sum_v \frac{\pi(v)^2}{\s(v)}\right) - 1} = \m_\s(\pi),
\end{equation}
where $\ind = (1,1,\dots,1)$.

Let $1 = \lam_1 \ge \lam_2\ge\dots\ge\lam_n\ge -1$ be the eigenvalues of $M$ , with a corresponding eigenbasis $\ind = f_1,\dots,f_n$, orthonormal with respect to the $\langle\cdot, \cdot\rangle_\s$ inner product. Then (see e.g.~\cite{LevinPeres17})
\al{
  \frac{\pi M(v)}{\s(v)} - 1 & = \sum_{j = 2}^n \sum_u \pi(u)f_j(u)f_j(v)\lam_j  \\
  & = \sum_{j=2}^n \lam_jf_j(v)\sum_u\left[\left(\frac{\pi(u)}{\s(u)} - 1\right)f_j(u)\s(u) + f_j(u)\s(u)\right] \\
  & = \sum_{j=2}^n \lam_jf_j(v)\left[\left\langle \frac{\pi}{\s} - \ind, f_j\right\rangle_\s + \langle f_j, \ind\rangle_\s\right].
}
For any $j > 1$, orthonormality implies $\langle f_j, \ind\rangle_\s = 0$. Let $F(j) = \langle \pi/\s-\ind, f_j\rangle_\s$, and note that $\m_\s(\pi)^2 = \sum_{j=1}^n F(j)^2$. Then
\al{
  \m_\s(\pi M)^2 & = \sum_v \s(v)\left(\sum_{j=2}^n \lam_j f_j(v)F(j)\right)^2 \\
  & = \sum_{j\ge 2} \lam_j^2F(j)^2 \|f_j\|_\s^2 + 2\sum_{k > j \ge 2} \lam_j\lam_kF(j)F(k) \langle f_j, f_k\rangle_\s
}
Since the $f_j$ are orthonormal in the $\langle\cdot,\cdot\rangle_\s$ inner product, we are left with
\al{
  \m_\s(\pi M)^2 & = \sum_{j\ge 2} \lam_j^2F(j)^2 \le \lam^2\sum_{j=1}^n F(j)^2 = \lam^2\m_\s(\pi)^2.
}
\end{proof}

\subsection{The expander mixing lemma}

\begin{proof}[Proof of Lemma~\ref{lem:EML}]
For $A\subseteq V$ let $\ind_A : V\to \{0,1\}$ be the indicator for $A$. Let $\pi(u) = \ind_{A} / |A|$. One easily checks that
\al{
  \frac{1}{|A|}M(A, B) - \s(B) = \left\langle\frac{\pi M(\cdot)}{\s(\cdot)} - \ind_V, \ind_B\right\rangle_\s.
}
Cauchy-Schwarz' inequality and Lemma~\ref{lem:norm} then give
\al{
  \left|\frac{1}{|A|}M(A, B) - \s(B)\right|^2 \le \left\|\frac{\pi M}{\s} - \ind_V\right\|_\s^2 \|\ind_B\|_\s^2 \le \lam\m_\s(\pi)^2 \s(B).
}
Since $\m_\s(\pi)^2 \le \sum_{u\in A}\pi(u)^2/\s(u) \le bn/|A|$, \eqref{eq:fullEML} follows.

To see how (i) follows, note that $R(u, v) \ge dM(u, v)$ for all $u, v$. Lemma~\ref{lem:sublinear} gives $\s(B) = \OM(1)$ whenever $|B| = \OM(n)$. Since $\lam(R) = o(1)$, for $|A|, |B| = \OM(n)$ we then have
$$
R(A, B) \ge dM(A, B) \ge d(|A|\s(B) - O(\lam\sqrt{n|A|\s(B)}) = \OM(dn).
$$

For (ii), let $c > 0$ be some constant to be chosen, and for each $|A| \le n/2$ define $A'$ as the set of $u\in A$ with $M(u, A) \ge (|A|/n)^c$. If $|A| < (n^c\|M\|)^{-1/(1-c)}$ then any $u\in A$ has
$$
M(u, A) \le \|M\||A| < \bfrac{|A|}{n}^c,
$$
so $A' = \emptyset$. Suppose $(n^c\|M\|)^{-1/(1-c)} \le |A| \le n/2$. Then \eqref{eq:fullEML} and the power law condition \eqref{eq:powerlaw} give, for some constant $0\le \a < 1/2$,
\begin{multline}
  \bfrac{|A|}{n}^c \le \frac{1}{|A'|} M(A', A) \le \s(A) + \lam\sqrt{\frac{bn\s(A)}{|A'|}} \\
  \le \left[\bfrac{|A|}{n}^{1-2\a-2c} + \lam\sqrt{\frac{b|A|}{|A'|}} \bfrac{n}{|A|}^{\a+2c}\right]\bfrac{|A|}{n}^{2c}.\label{eq:EMLstuff}
\end{multline}
For $2c < 1-2\a$, the first term in square brackets is at most 1. Since $\lam \le (n\|M\|)^{-\a-\g}$ for some constant $\g > 0$, we have for $|A| \ge (n^c\|M\|)^{-1/(1-c)}$ that
\al{
  \bfrac{n}{|A|}^{\a+2c} & \le \bfrac{n}{n^{-\frac{c}{1-c}}\|M\|^{-\frac{1}{1-c}}}^{\a+2c} = (n\|M\|)^{\frac{\a+2c}{1-c}}.
}
We have $n\|M\| \ge 1$ since $M$ is a transition matrix. Since $\lam = o(1)$ and $\lam \le (n\|M\|)^{-\a-\g}$ for some constant $\g > 0$, we conclude that for $c > 0$ small enough, $\lam(n/|A|)^{\a+2c} = o(1)$. From \eqref{eq:EMLstuff} we then have, for $|A| \le n/2$,
\al{
  1 < 2^c \le \bfrac{n}{|A|}^c \le 1 + o\left(\sqrt{\frac{b|A|}{|A'|}}\right).
}
We conclude that $|A'| = o(|A|)$.
\end{proof}

\bibliographystyle{plain}
\bibliography{201222_gnpi}

\begin{thebibliography}{10}

\bibitem{AjtaiKomlosSzemeredi}
M.~Ajtai, J.~Koml\'os, and E.~Szemer\'edi.
\newblock First occurrence of {H}amilton cycles in random graphs.
\newblock {\em Annals of Discrete Mathematics}, 115:173--178, 1985.

\bibitem{AlonChung88}
N.~Alon and F.R.K. Chung.
\newblock Explicit construction of linear sized tolerant networks.
\newblock {\em Discrete Mathematics}, 72(1):15 -- 19, 1988.

\bibitem{AlonKrivelevich20}
Yahav Alon and Michael Krivelevich.
\newblock Hitting time of edge disjoint {H}amilton cycles in random subgraph
  processes on dense base graphs.
\newblock {\em ArXiv e-prints}, 2020.

\bibitem{AnastosFriezeGao19}
Michael Anastos, Alan Frieze, and Pu~Gao.
\newblock Hamiltonicity of random graphs in the stochastic block model.
\newblock {\em ArXiv e-prints}, 2019.

\bibitem{bollobas84}
B.~Bollob\'as.
\newblock The evolution of random graphs.
\newblock {\em Transactions of the American Mathematical Society},
  286(1):257--274, 1984.

\bibitem{BollobasFrieze}
B.~Bollob{\'a}s and A.~M. Frieze.
\newblock On matchings and {H}amiltonian cycles in random graphs.
\newblock {\em Annals of Discrete Mathematics}, 28:23--46, 1985.

\bibitem{Durrett10}
Rick Durrett.
\newblock {\em Probability: Theory and Examples}.
\newblock Cambridge Series in Statistical and Probabilistic Mathematics.
  Cambridge University Press, 4 edition, 2010.

\bibitem{ErdosRenyi60}
P.~Erd{\H{o}}s and A~R{\'{e}}nyi.
\newblock On the evolution of random graphs.
\newblock {\em Publ. Math. Inst. Hungar. Acad. Sci.}, 5:17--61, 1960.

\bibitem{Frieze85}
A.~M. Frieze.
\newblock Limit distribution for the existence of {H}amiltonian cycles in
  random bipartite graphs.
\newblock {\em Europ. J. Combinatorics}, 6:327--334, 1985.

\bibitem{Frieze19}
A.~M. Frieze.
\newblock Hamilton cycles in random graphs: a bibliography.
\newblock {\em ArXiv e-prints}, 1901.07139 [v13], July 2019.

\bibitem{FriezeJohansson17}
A.~M. Frieze and T.~Johansson.
\newblock On random $k$-out subgraphs of large graphs.
\newblock {\em Random Structures \& Algorithms}, 50(2):143--157, 2017.

\bibitem{FriezeKaronski}
A.~M. Frieze and M.~Karo\'nski.
\newblock {\em Introduction to Random Graphs}.
\newblock Cambridge University Press, Cambridge, UK, 2015.

\bibitem{FriezeKrivelevich02}
Alan Frieze and Michael Krivelevich.
\newblock Hamilton cycles in random subgraphs of pseudo-random graphs.
\newblock {\em Discrete Mathematics}, 256:137--150, September 2002.

\bibitem{Johansson20}
T.~Johansson.
\newblock On {H}amilton cycles in {E}rd{\H{o}}s-{R}{\'{e}}nyi subgraphs of
  large graphs.
\newblock {\em Random Structures \& Algorithms}, 57:132--149, 2020.

\bibitem{KomlosSzemeredi}
J.~Koml\'os and E.~Szemer\'edi.
\newblock Limit distribution for the existence of {H}amiltonian cycles in a
  random graph.
\newblock {\em Discrete Mathematics}, 43(1):55--63, 1983.

\bibitem{LevinPeres17}
David~A Levin and Yuval Peres.
\newblock {\em Markov chains and mixing times}, volume 107.
\newblock American Mathematical Soc., 2017.

\bibitem{Montgomery19}
Richard Montgomery.
\newblock Hamiltonicity in random graphs is born resilient.
\newblock {\em Journal of Combinatorial Theory, Series B}, 139:316 -- 341,
  2019.

\bibitem{Nilli91}
A.~Nilli.
\newblock On the second eigenvalue of a graph.
\newblock {\em Discrete Mathematics}, 91(2):207 -- 210, 1991.

\end{thebibliography}

\end{document}